\pdfoutput=1
\documentclass[12pt,reqno]{amsart}

\usepackage{amsmath,amsthm,amssymb,amsfonts,mathtools}
\usepackage{bm}
\usepackage{enumitem}
\usepackage{graphicx,latexsym}
\usepackage{algorithm,algorithmic}
\usepackage{booktabs}
\usepackage{array}
\usepackage{url}
\numberwithin{equation}{section}

\makeatletter
\def\subsection{\@startsection{subsection}{2}{\z@}%
  {.7\linespacing\@plus.3\linespacing}{.35\linespacing}%
  {\normalfont\bfseries\def\@secnumfont{\bfseries}}}
\def\paragraph{\@startsection{paragraph}{4}{\z@}%
  {.45\linespacing\@plus.2\linespacing}{.2\linespacing}%
  {\normalfont\bfseries}}
\makeatother

\makeatletter
\def\@captionfont{\scriptsize}
\makeatother

\makeatletter
\def\@setabstracta{%
  \ifvoid\abstractbox
  \else
    \skip@20\p@ \advance\skip@-\lastskip
    \advance\skip@-\baselineskip \vskip\skip@
    \unvbox\abstractbox
    \prevdepth\z@
  \fi
}
\makeatother

\theoremstyle{plain}
\newtheorem{theorem}{Theorem}[section]
\newtheorem{lemma}[theorem]{Lemma}
\newtheorem{proposition}[theorem]{Proposition}
\newtheorem{corollary}[theorem]{Corollary}
\theoremstyle{definition}
\newtheorem{definition}[theorem]{Definition}
\newtheorem{assumption}[theorem]{Assumption}

\newcommand{\diagop}{\operatorname{diag}}
\newcommand{\rankop}{\operatorname{rank}}
\newcommand{\distop}{\operatorname{dist}}
\newcommand{\prox}{\operatorname{prox}}
\newcommand{\Lip}{\operatorname{Lip}}
\newcommand{\argmin}{\operatorname*{arg\,min}}
\newcommand{\R}{\mathbb{R}}
\newcommand{\E}{\mathbb{E}}

\newcommand{\norm}[1]{\left\lVert #1\right\rVert}
\newcommand{\abs}[1]{\left\lvert #1\right\rvert}
\newcommand{\inner}[2]{\left\langle #1,#2\right\rangle}
\newcommand{\st}{\;:\;}
\newcommand{\cA}{\mathcal{A}}
\newcommand{\cE}{\mathcal{E}}

\title[Smooth-to-Hard Spectral Regularization]{Data-Driven Smooth-to-Hard Spectral Regularization for General Linear Inverse Problems}
\author[Yang, Wang, and Yang]{Guanzhong Yang}
\address{Department of Mathematics\\ Imperial College London, UK}
\email{guanzhong.yang25@imperial.ac.uk}

\author{Zi Wang}
\address{Department of Bioengineering\\ Imperial-X, Imperial College London, UK}
\email{zi.wang@imperial.ac.uk}

\author{Guang Yang}
\address{Department of Bioengineering\\ Imperial-X, Imperial College London, UK}
\email{g.yang@imperial.ac.uk}
\thanks{Zi Wang and Guang Yang are co-corresponding authors.}

\begin{document}

\begin{abstract}
We study low-rank recovery from noisy linear observations \(y=\cA(X_0)+\varepsilon\), where \(X_0\in\R^{M\times N}\) is unknown and \(\cA:\R^{M\times N}\to\R^p\) is a known linear measurement operator.  The operator need not be injective on the full matrix space; whenever prediction error is converted into Frobenius reconstruction error, we assume that \(\cA\) satisfies a restricted stability lower bound on the relevant structured error class.  Nuclear-norm regularization is convex, but soft thresholding imposes a persistent absolute bias on retained singular values; rank regularization avoids this bias, whereas its proximal map is discontinuous hard thresholding.  Building on an earlier SURE analysis of the same shrinkage family for direct matrix denoising, we investigate the variational and optimization properties of the logistic smooth-to-hard shrinker
\[
    f_{a,\lambda}(s)
    =\frac{s}{1+\exp[-a(s/\lambda-1)]},\qquad s\geq0,
\]
where \(\lambda\) locates the transition and \(a\) controls its sharpness.  Our principal contribution concerns this particular family: we derive an exact Lambert-\(W\) inverse, construct the induced nonconvex penalty, and identify its exact weak-convexity constant \(\rho_a=1-L_a^{-1}\), where \(L_a\) is the maximum shrinker slope.  This identity unifies worst-case sensitivity, negative curvature, and proximal strong-convexity margin.  Standard spectral proximal principles then lift the construction to singular values.  In the recovery algorithm the required proximal scale is \(\beta=\eta\alpha\), where \(\eta>0\) is the proximal-gradient step size and \(\alpha>0\) is the regularization weight; the scaled matrix proximal map is unique whenever \(\beta\rho_a<1\).  Specializing established nonconvex proximal-gradient arguments gives sufficient descent, a vanishing stationarity residual, stationary cluster points, and an explicit finite-iteration bound.  Independent validation gives fixed-candidate conditional unbiasedness and a finite-grid oracle inequality, while restricted stability converts prediction error into Frobenius error.  A retrospective single-image evaluation on \(20\) OCMR series covered accelerations \(R\in\{4,6,8\}\), nominal SNR levels \(\{0,10,20,30\}\) dB, and Soft, Hard, and L+S baselines.  The proposed rule achieved higher mean PSNR than Hard in \(11\) of \(12\) conditions and showed a noise-dependent crossover relative to Soft and L+S, with higher mean PSNR than all three baselines at every \(30\)-dB condition.
\end{abstract}

\maketitle

\section{Introduction}

Low-rank structure is widely used in matrix completion, dynamic imaging, spatio-temporal inverse problems, and collections of related signals.  We consider
\[
    \cA:\R^{M\times N}\to\R^p,
    \qquad y=\cA(X_0)+\varepsilon.
\]
Because \(\cA\) may be undersampled and non-injective, image-domain guarantees require restricted rather than global stability: for a relevant structured error class \(\cE\),
\begin{equation}\label{eq:restricted-stability-intro}
    \frac1p\norm{\cA(H)}_2^2\geq\kappa\norm{H}_F^2,
    \qquad H\in\cE.
\end{equation}
A direct formulation penalises rank,
\begin{equation}\label{eq:rank-model}
    \min_X\;
    \frac12\norm{\cA(X)-y}_2^2+\gamma\rankop(X),
\end{equation}
whereas nuclear-norm regularization uses
\[
  \norm{X}_*=\sum_i\sigma_i(X)
\]
and yields singular-value soft thresholding \cite{recht2010,cai2010}.  Hard thresholding is nearly unbiased above its threshold but discontinuous; soft thresholding is continuous but imposes a constant absolute bias.  Neither geometry is fully satisfactory when small components should be suppressed without systematically attenuating dominant ones.  Motivated by nonconvex bias-reduction principles such as MCP \cite{zhang2010}, we seek a continuous spectral rule that approaches hard thresholding while retaining a precise variational and algorithmic interpretation.

We study the logistic shrinker
\[
    f_{a,\lambda}(s)
    =\frac{s}{1+\exp[-a(s/\lambda-1)]},\qquad s\geq0,
\]
where \(\lambda\) locates the transition and \(a\) controls its sharpness.  This family was previously used for SURE-based direct matrix denoising \cite{yang2026}; here the questions are variational and algorithmic.  We invert the map exactly using the Lambert \(W\) function, construct its induced saturating penalty, and establish the hard-threshold limit.  If \(L_a\) is the maximum normalized shrinker slope, the same calculation gives
\begin{equation}\label{eq:tradeoff-intro}
    \boxed{
    \begin{aligned}
    \text{worst input sensitivity}&=L_a,\\
    \text{weak convexity}&=1-L_a^{-1},\\
    \text{unit-prox curvature margin}&=L_a^{-1}.
    \end{aligned}}
\end{equation}
Thus increased hardness reduces large-signal bias but increases worst-case sensitivity and weak nonconvexity.

The inverse construction is important because specifying a shrinker does not by itself identify an objective function or justify its use inside an iterative recovery method.  Here it produces an explicit penalty, reveals its exact negative curvature, and distinguishes the designed unit-scale shrinker from the generally scaled proximal map required by recovery algorithms.

The scalar penalty is lifted to singular values without changing the weak-convexity constant.  At proximal scale \(\beta=\eta\alpha\), each singular value is obtained from a monotone scalar equation and the update is unique when \(\beta\rho_a<1\).  The resulting proximal-gradient method has sufficient descent, vanishing stationarity residuals, stationary cluster points, and an explicit finite-iteration bound.  These guarantees concern stationary behaviour rather than global optimality, and boundedness is stated separately because a saturating penalty need not make an undersampled objective coercive.

Validation on a finite candidate grid gives an oracle inequality for prediction risk.  This statistical statement is kept separate from the optimization analysis: it controls the split-trained candidates in the validation measurement domain, whereas restricted stability is additionally required to obtain Frobenius reconstruction control.  A subsequent full-data refit is an empirical post-selection step rather than an automatic consequence of the oracle inequality.

Accordingly, the paper distinguishes three levels of analysis.  The variational results identify the penalty represented by the shrinker; the optimization results describe a fixed-parameter algorithm; and the validation results compare a prespecified finite collection of training-based reconstructions.  None of these levels alone implies global recovery for an undersampled operator.

The numerical study uses \(20\) OCMR cardiac-cine series, accelerations \(R\in\{4,6,8\}\), and nominal SNR levels from \(0\) to \(30\) dB.  Soft, Hard, and L+S comparisons show an advantage over Hard in most conditions and a noise-dependent trade-off against Soft and L+S.

\subsection*{Related work and scope of the contributions}

Nonconvex singular-value penalties and generalized singular-value thresholding are established tools for reducing convex shrinkage bias \cite{lu2014,lu2015}.  Inverse construction of penalties from shrinkage maps is also known \cite{chartrand2014}, and logistic singular-value shaping has appeared in rank estimation \cite{yadav2015}.  Those works establish the broader principles, but they do not provide the present family-specific combination of an exact inverse, sharp weak-convexity constant, arbitrary-scale proximal equation, and recovery analysis under general linear measurements.  The proximal-gradient and finite-grid validation results nevertheless specialize established frameworks \cite{gong2013} rather than propose new general theories.

Within these boundaries, the contributions of the paper are the following.
\begin{enumerate}[label=\textup{(\roman*)}]
    \item We derive the exact inverse and induced saturating penalty, prove that the shrinker is its unique unit-scale proximal map, and identify the hard-threshold limits.
    \item We prove the sharp weak-convexity identity \(\rho_a=1-L_a^{-1}\) for both scalar and spectral penalties and derive the scaled proximal equation and uniqueness condition \(\beta\rho_a<1\).
    \item We apply standard optimization and validation guarantees to the inverse problem, keeping prediction and reconstruction assumptions distinct.
\end{enumerate}

The spectral lifting uses standard facts for unitarily invariant functions \cite{lewis1995,lewissendov2005}; all shrinker-specific inverse, curvature, scaling, and limiting calculations are proved below.

\section{Low-Rank Inverse Problems and the Regularization Model}

\subsection{Observation model}

Let \(X_0\in\R^{M\times N}\) be unknown and suppose
\begin{equation}\label{eq:observation-model}
    y=\cA(X_0)+\varepsilon\in\R^p,
\end{equation}
where \(\cA:\R^{M\times N}\to\R^p\) is a known linear measurement operator, including missing-entry maps, convolutional blur, downsampling, and subsampled Fourier measurements.  In contrast with direct denoising, \(\cA\) need not be injective, and \(\cA^*\cA\) need not be a multiple of the identity.  We restrict the scope of the paper to observation designs that are quantitatively injective on the structured error classes relevant to reconstruction.

\begin{assumption}[Restrictedly stable observation design]\label{ass:restricted-stability}
The full measurement operator \(\cA\) belongs to a class for which there
exist a prescribed structured error set
\(\cE\subseteq\R^{M\times N}\) and a constant \(\kappa>0\) satisfying
\begin{equation}\label{eq:restricted-stability}
    \frac1p\norm{\cA(H)}_2^2
    \geq\kappa\norm{H}_F^2,
    \qquad H\in\cE.
\end{equation}
When the independent validation split of Section~8 is used, the chosen
validation block
\(\cA_{\mathrm{val}}:\R^{M\times N}\to\R^m\) is likewise restricted to
operators for which there exist a prescribed validation error class
\(\cE_{\mathrm{val}}\subseteq\R^{M\times N}\) and
\(\kappa_{\mathrm{val}}>0\) satisfying
\begin{equation}\label{eq:validation-restricted-stability}
    \frac1m\norm{\cA_{\mathrm{val}}(H)}_2^2
    \geq\kappa_{\mathrm{val}}\norm{H}_F^2,
    \qquad H\in\cE_{\mathrm{val}}.
\end{equation}
\end{assumption}

Here ``prescribed'' means that the structural class is specified independently of the validation noise.  In Section~8 the validation statement is applied conditionally on the training data, and the realized candidate errors are required to lie in \(\cE_{\mathrm{val}}\).  Equivalently, one may choose in advance a deterministic structural superset containing all candidate errors under consideration.

Assumption \ref{ass:restricted-stability} restricts the class of
observation designs studied in this paper; it does not assert that an
arbitrary linear operator satisfies either lower bound.  In particular,
restricted stability of the full operator \(\cA\) does not automatically
pass to an arbitrary validation subblock, which is why the validation
design is included explicitly in the same assumption.  The lower bound is
equivalent to requiring a positive minimum singular value after restricting
the operator to the normalized directions in the relevant error class.  It
is therefore necessary that
\(\cE\cap\ker\cA=\{0\}\), with an analogous requirement for
\(\cE_{\mathrm{val}}\) and \(\cA_{\mathrm{val}}\).  Such conditions are
standard in structured inverse problems: random matrix-sensing operators
can satisfy them on low-rank error sets when sufficiently many measurements
are available \cite{recht2010}.
Thus the assumption is realistic for appropriate measurement designs but
is not intended to cover arbitrary undersampled operators.

The matrix space is equipped with
\[
    \inner{X}{Z}_F=\operatorname{tr}(X^T Z),
    \qquad \norm{X}_F^2=\inner{X}{X}_F.
\]
The adjoint \(\cA^*:\R^p\to\R^{M\times N}\) is defined by
\(
\inner{\cA(X)}{z}_{\R^p}=\inner{X}{\cA^*z}_F
\).
Set
\begin{equation}\label{eq:operator-constant}
    \norm{\cA}_{\mathrm{op}}
    =\sup_{\norm{X}_F=1}\norm{\cA(X)}_2,
    \qquad L_{\cA}=\norm{\cA}_{\mathrm{op}}^2.
\end{equation}
Gaussian noise, \(\varepsilon\sim N(0,\tau^2I_p)\), will be imposed only in the validation analysis.  The deterministic optimization results require no distributional assumption.

Write \(q=\min(M,N)\) and let
\[
    X=U\diagop(\sigma_1(X),\ldots,\sigma_q(X))V^T
\]
be a thin singular value decomposition, with
\(
\sigma_1(X)\geq\cdots\geq\sigma_q(X)\geq0
\).

\subsection{Why a new spectral shrinker is needed}

For \(s\geq0\), the scalar rank penalty gives hard thresholding,
\begin{equation}\label{eq:rank-prox-scalar}
  \prox_{\gamma\mathbf{1}_{\{\cdot\neq0\}}}(s)
  =
  \begin{cases}
    0, & s<\sqrt{2\gamma},\\
    s, & s>\sqrt{2\gamma},
  \end{cases}
\end{equation}
with both values minimizing at equality.  The nuclear norm gives
\begin{equation}\label{eq:nuclear-prox}
    \prox_{\gamma\norm{\cdot}_*}(Y)
    =U\diagop\bigl((s_i-\gamma)_+\bigr)V^T
\end{equation}
for \(Y=U\diagop(s_i)V^T\).  We therefore seek a continuous increasing map that suppresses small singular values, has vanishing relative bias above a scale \(\lambda\), approaches hard thresholding as its hardness grows, and remains the proximal map of a well-defined spectral penalty.

\subsection{Regularization model}

Later, we will construct a scalar penalty \(\phi_{a,\lambda}:[0,\infty)\to[0,\infty)\) and its spectral lift
\begin{equation}\label{eq:spectral-penalty-def-preview}
    \Phi_{a,\lambda}(X)
    =\sum_{i=1}^q\phi_{a,\lambda}(\sigma_i(X)).
\end{equation}
Thus, for fixed \(a,\lambda,\alpha>0\), the recovery objective is
\begin{equation}\label{eq:main-objective}
    F_{a,\lambda}(X)
    =g(X)+\alpha\Phi_{a,\lambda}(X),
    \qquad
    g(X)=\frac12\norm{\cA(X)-y}_2^2.
\end{equation}
The threshold \(\lambda\) has the same units as a singular value, whereas the hardness \(a\) is dimensionless.  The coefficient \(\alpha\) balances the data term and the spectral penalty.  It is part of the statistical model and must be distinguished from the algorithmic step size introduced later.

The penalty in \eqref{eq:spectral-penalty-def-preview} will not be postulated independently.  Instead, we first prescribe a scalar shrinkage rule with the desired smooth-to-hard behaviour and then recover the penalty for which that rule is the unit-scale proximal map.  This inverse-design step is carried out next.

\section{Scalar Shrinkage and Variational Construction}

\subsection{The logistic shrinker}

For \(a>0\), define the logistic gate and normalized shrinker
\begin{equation}\label{eq:normalized-shrinker}
    \ell_a(t)=\frac{1}{1+e^{-a(t-1)}},
    \qquad
    h_a(t)=t\ell_a(t),
    \qquad t\geq0.
\end{equation}
For a threshold scale \(\lambda>0\), set
\begin{equation}\label{eq:scaled-shrinker}
    f_{a,\lambda}(s)
    =\lambda h_a(s/\lambda)
    =\frac{s}{1+\exp[-a(s/\lambda-1)]},
    \qquad s\geq0.
\end{equation}

The normalized representation separates the two roles of the parameters: \(\lambda\) fixes the location of the transition, while \(a\) changes its shape without changing its physical scale.  We first establish the geometric properties needed for the subsequent inverse construction.

\subsection{Monotonicity, bias, and the hard-threshold limit}

\begin{lemma}[Strict monotonicity]\label{lem:monotone}
For every \(a>0\), the map \(h_a:[0,\infty)\to[0,\infty)\) is a strictly increasing bijection.  Consequently, \(f_{a,\lambda}\) is a strictly increasing bijection for every \(\lambda>0\).
\end{lemma}

\begin{proof}
Differentiating \(\ell_a\) gives
\[
    \ell_a'(t)=a\ell_a(t)(1-\ell_a(t)).
\]
Hence
\begin{equation}\label{eq:h-prime}
    h_a'(t)
    =\ell_a(t)+at\ell_a(t)(1-\ell_a(t))>0
\end{equation}
for every \(t\geq0\).  Thus \(h_a\) is strictly increasing and continuous.  Moreover, \(h_a(0)=0\) and \(h_a(t)\to\infty\) because \(\ell_a(t)\to1\).  It is therefore onto \([0,\infty)\).  Positive rescaling yields the claim for \(f_{a,\lambda}\).
\end{proof}

The bijectivity established above is what makes the later inverse-penalty construction possible.  Before using that inverse, we verify that the proposed map has the intended shrinkage bias and approaches the two branches of hard thresholding away from the transition point.

\begin{proposition}[Shrinkage and exponential approximation away from the threshold]\label{prop:bias}
For every \(s>0\),
\begin{equation}\label{eq:strict-shrinkage}
    0<f_{a,\lambda}(s)<s.
\end{equation}
The exact deviation from the identity branch is
\begin{equation}\label{eq:bias-exact}
    s-f_{a,\lambda}(s)
    =\frac{s}{1+\exp[a(s/\lambda-1)]}.
\end{equation}
For every \(\delta>0\), the approximation error is exponentially small on both sides of the transition.  If \(0\leq s\leq\lambda-\delta\), then
\begin{equation}\label{eq:below-threshold-bound}
    0\leq f_{a,\lambda}(s)
    \leq s e^{-a\delta/\lambda},
\end{equation}
whereas if \(s\geq\lambda+\delta\), then
\begin{equation}\label{eq:bias-bound}
    0\leq s-f_{a,\lambda}(s)
    \leq s e^{-a\delta/\lambda}.
\end{equation}
\end{proposition}

\begin{proof}
Since \(0<\ell_a<1\), strict shrinkage follows directly.  Subtracting \eqref{eq:scaled-shrinker} from \(s\) gives \eqref{eq:bias-exact}.

Suppose first that \(0\leq s\leq\lambda-\delta\).  Then
\[
    1-\frac{s}{\lambda}\geq\frac{\delta}{\lambda},
\]
and hence
\[
    f_{a,\lambda}(s)
    =\frac{s}{1+\exp[a(1-s/\lambda)]}
    \leq\frac{s}{\exp(a\delta/\lambda)}
    =s e^{-a\delta/\lambda},
\]
which proves \eqref{eq:below-threshold-bound}.  If instead \(s\geq\lambda+\delta\), then
\(
s/\lambda-1\geq\delta/\lambda
\),
so the denominator in \eqref{eq:bias-exact} is at least \(e^{a\delta/\lambda}\).  This gives \eqref{eq:bias-bound}.
\end{proof}

\begin{proposition}[Hard-threshold limit]\label{prop:hard-limit-shrinker}
For every fixed \(s\neq\lambda\),
\begin{equation}\label{eq:hard-limit-shrinker}
    f_{a,\lambda}(s)\longrightarrow
    H_\lambda(s)
    :=\begin{cases}
        0,&0\leq s<\lambda,\\
        s,&s>\lambda,
    \end{cases}
    \qquad a\to\infty.
\end{equation}
At the transition point, \(f_{a,\lambda}(\lambda)=\lambda/2\) for every finite \(a\).
\end{proposition}

\begin{proof}
If \(s<\lambda\), then \(-a(s/\lambda-1)\to+\infty\), so the denominator in \eqref{eq:scaled-shrinker} diverges.  If \(s>\lambda\), the exponential term tends to zero.  The value at \(s=\lambda\) is immediate.
\end{proof}

For every finite \(a\), however, \(f_{a,\lambda}(s)>0\) whenever \(s>0\).  The proposed family therefore produces effective rather than exact low rank: small singular values may become negligible, but they are not set identically to zero.  Nor is the family an exact interpolation from soft thresholding; its finite-\(a\) members form a distinct class of smooth shrinkers whose limiting behaviour is hard thresholding.

Strict monotonicity now supplies the key link from a shrinkage rule to a regularizer.  Because \(f_{a,\lambda}\) is invertible, its inverse can be inserted into the first-order condition of a proximal problem.  The required inverse is explicit in terms of the Lambert function, defined by \(W(z)e^{W(z)}=z\).  All arguments below are nonnegative, so only the real principal branch \(W_0\) is needed \cite{corless1996}.

\subsection{Exact inverse and induced penalty}

\begin{proposition}[Inverse formula]\label{prop:inverse}
For every \(v\geq0\),
\begin{equation}\label{eq:h-inverse}
    h_a^{-1}(v)
    =v+\frac1aW_0\!\left(av e^{a(1-v)}\right).
\end{equation}
Consequently, for \(u\geq0\),
\begin{equation}\label{eq:f-inverse}
    f_{a,\lambda}^{-1}(u)
    =u+\frac{\lambda}{a}
    W_0\!\left(\frac{au}{\lambda}e^{a(1-u/\lambda)}\right).
\end{equation}
\end{proposition}

\begin{proof}
Let \(v=h_a(t)\).  Then
\[
    v+ve^{-a(t-1)}=t,
    \qquad t-v=ve^{-a(t-1)}.
\]
With \(r=t-v\),
\[
    r=ve^{a(1-v)}e^{-ar},
    \qquad (ar)e^{ar}=av e^{a(1-v)}.
\]
Applying \(W_0\) yields \eqref{eq:h-inverse}.  If \(u=f_{a,\lambda}(s)=\lambda h_a(s/\lambda)\), apply \eqref{eq:h-inverse} to \(u/\lambda\) and multiply by \(\lambda\).
\end{proof}

The argument of \(W_0\) in \eqref{eq:f-inverse} is dimensionless because both \(a\) and \(u/\lambda\) are dimensionless, while the prefactor \(\lambda/a\) restores the units of a singular value.  We next use this inverse to identify the regularizer rather than postulating a penalty independently.

\begin{definition}[Induced penalty]\label{def:penalty}
Define
\begin{equation}\label{eq:normalized-penalty}
    p_a(v)=\int_0^v\bigl(h_a^{-1}(r)-r\bigr)\,dr,
    \qquad v\geq0,
\end{equation}
and
\begin{equation}\label{eq:scaled-penalty}
    \phi_{a,\lambda}(u)
    =\lambda^2p_a(u/\lambda),
    \qquad u\geq0.
\end{equation}
The lower integration limit fixes \(\phi_{a,\lambda}(0)=0\).
\end{definition}

By differentiation and Proposition \ref{prop:inverse},
\begin{align}
    \phi_{a,\lambda}'(u)
    &=f_{a,\lambda}^{-1}(u)-u \label{eq:phi-prime-inverse}\\
    &=\frac{\lambda}{a}
    W_0\!\left(\frac{au}{\lambda}e^{a(1-u/\lambda)}\right).
    \label{eq:phi-prime-lambert}
\end{align}

\begin{theorem}[Unique unit-scale proximal representation]\label{thm:unit-prox-scalar}
For every \(a,\lambda>0\) and \(s\geq0\),
\begin{equation}\label{eq:unit-prox-scalar}
    \prox_{\phi_{a,\lambda}}(s)
    :=\argmin_{u\geq0}
    \left\{\frac12(u-s)^2+\phi_{a,\lambda}(u)\right\}
    =f_{a,\lambda}(s).
\end{equation}
The minimizer is unique.
\end{theorem}

\begin{proof}
Let
\(
G_s(u)=\tfrac12(u-s)^2+\phi_{a,\lambda}(u)
\).
For \(u>0\), \eqref{eq:phi-prime-inverse} gives
\begin{equation}\label{eq:unit-objective-derivative}
    G_s'(u)=f_{a,\lambda}^{-1}(u)-s.
\end{equation}
Hence the only stationary point is \(u=f_{a,\lambda}(s)\).  Differentiating the identity
\(
f_{a,\lambda}(f_{a,\lambda}^{-1}(u))=u
\)
gives
\begin{equation}\label{eq:unit-objective-curvature}
    G_s''(u)
    =(f_{a,\lambda}^{-1})'(u)
    =\frac{1}{f_{a,\lambda}'(f_{a,\lambda}^{-1}(u))}>0.
\end{equation}
Thus \(G_s\) is strictly convex.  If \(s=0\), then \(G_0'(u)=f_{a,\lambda}^{-1}(u)>0\) for \(u>0\), so the unique minimizer is the boundary point \(0=f_{a,\lambda}(0)\).
\end{proof}

Theorem \ref{thm:unit-prox-scalar} converts the original shrinkage design into a genuine variational model.  It remains to understand the shape of the resulting penalty and to verify that its limiting behaviour matches the hard-threshold motivation.

\subsection{Geometry and hard-threshold limit of the penalty}

\begin{proposition}[Penalty geometry]\label{prop:penalty-shape}
The penalty \(\phi_{a,\lambda}\) is nonnegative and strictly increasing on \((0,\infty)\).  Moreover,
\begin{equation}\label{eq:phi-derivative-limits}
    \phi_{a,\lambda}'(0)=0,
    \qquad
    \phi_{a,\lambda}'(u)>0\;(u>0),
    \qquad
    \phi_{a,\lambda}'(u)\to0\;(u\to\infty).
\end{equation}
In fact, \(\phi_{a,\lambda}\) has a finite limiting value as \(u\to\infty\).
\end{proposition}

\begin{proof}
Strict shrinkage implies that the inverse exceeds its positive input.  Equation \eqref{eq:phi-prime-inverse} therefore proves monotonicity and nonnegativity.  The value at zero follows from \(f_{a,\lambda}^{-1}(0)=0\).

For saturation, write \(v=u/\lambda\) and \(t=h_a^{-1}(v)\).  When \(v\geq1\), one has \(t>v\geq1\), so \(\ell_a(t)\geq1/2\) and therefore \(t\leq2v\).  Hence
\[
    0<t-v
    =\frac{t}{1+e^{a(t-1)}}
    \leq t e^{-a(t-1)}
    \leq2v e^{-a(v-1)}.
\]
The final bound is integrable over \([1,\infty)\).  Since \(p_a'(v)=t-v\), the derivative tends to zero and \(p_a(v)\), hence \(\phi_{a,\lambda}(u)\), converges to a finite limit.
\end{proof}

The curvature is
\begin{equation}\label{eq:phi-second}
    \phi_{a,\lambda}''(u)
    =\frac{1}{f_{a,\lambda}'(f_{a,\lambda}^{-1}(u))}-1.
\end{equation}
Thus the penalty is locally concave exactly where the shrinker slope exceeds one.  This observation connects the variational construction to optimization: the largest shrinker slope will determine the most negative curvature that an algorithm must overcome.  Before quantifying that curvature, we complete the link with rank regularization by passing to the hard limit.

\begin{theorem}[Pointwise hard limit of the penalty]\label{thm:hard-penalty-limit}
For every fixed \(u\geq0\),
\begin{equation}\label{eq:hard-penalty-limit}
    \phi_{a,\lambda}(u)\longrightarrow
    \phi_{\infty,\lambda}(u)
    :=\begin{cases}
       \lambda u-\tfrac12u^2,&0\leq u\leq\lambda,\\
       \tfrac12\lambda^2,&u\geq\lambda,
    \end{cases}
    \qquad a\to\infty.
\end{equation}
\end{theorem}

\begin{proof}
Fix \(v>0\), let \(t_a=h_a^{-1}(v)\), and note that \(t_a>v\).  The sequence is bounded: if \(t_a<1\), then \(t_a\leq1\); if \(t_a\geq1\), then
\(
v=t_a\ell_a(t_a)\geq t_a/2
\),
so \(t_a\leq2v\).

If \(0<v<1\), any subsequential limit below one would force \(\ell_a(t_a)\to0\) and hence \(v=t_a\ell_a(t_a)\to0\); any limit above one would force \(\ell_a(t_a)\to1\) and hence a limit equal to \(v>1\).  Both are impossible, so \(t_a\to1\).  If \(v=1\), the same argument gives \(t_a\to1\).  If \(v>1\), then \(t_a>v>1\), so \(\ell_a(t_a)\to1\) and \(t_a\to v\).  Therefore, for every \(v>0\),
\begin{equation}\label{eq:inverse-gap-limit}
    h_a^{-1}(v)-v\longrightarrow(1-v)_+.
\end{equation}
At \(v=0\), the left side is instead identically zero; this single endpoint does not affect any integral.

For fixed \(V<\infty\) and \(0\leq r\leq V\), the preceding bound gives
\(
0\leq h_a^{-1}(r)-r\leq\max\{1,2V\}
\).
Dominated convergence, using \eqref{eq:inverse-gap-limit} almost everywhere on \([0,V]\), yields
\[
    p_a(v)\to\int_0^v(1-r)_+\,dr
    =\begin{cases}
       v-\tfrac12v^2,&0\leq v\leq1,\\
       \tfrac12,&v\geq1.
    \end{cases}
\]
Substitution into \eqref{eq:scaled-penalty} gives \eqref{eq:hard-penalty-limit}.
\end{proof}

Pointwise convergence of the penalties alone does not determine the minimizers of the corresponding proximal objectives.  To identify the variational meaning of the limiting penalty, we therefore solve its scalar proximal problem directly.

\begin{proposition}[Proximal map of the limiting penalty]\label{prop:hard-limit-prox}
If \(s\neq\lambda\), then
\(
\prox_{\phi_{\infty,\lambda}}(s)=H_\lambda(s)
\).
If \(s=\lambda\), every \(u\in[0,\lambda]\) is a minimizer.
\end{proposition}

\begin{proof}
For \(0\leq u\leq\lambda\),
\[
    \frac12(u-s)^2+\phi_{\infty,\lambda}(u)
    =\frac12s^2+(\lambda-s)u,
\]
which is affine in \(u\).  For \(u\geq\lambda\), the objective is
\(
\tfrac12(u-s)^2+\tfrac12\lambda^2
\).
Comparing the two pieces gives the stated alternatives.
\end{proof}

The scalar construction is therefore complete.  It begins with a continuous shrinker, produces an explicit saturating penalty, and recovers hard thresholding only in the singular limit \(a\to\infty\).  The next question is how much stability is lost as this limit is approached.

\section{Hardness, Sensitivity, and Exact Weak Convexity}\label{sec:weak-convexity}

\subsection{Maximum slope}

The hard-threshold limit reduces bias, but Proposition \ref{prop:hard-limit-shrinker} alone does not describe its stability cost.  That cost is governed by the steepest part of the transition.  We therefore define
\begin{equation}\label{eq:La-def}
    L_a=\sup_{t\geq0}h_a'(t).
\end{equation}
Because scaling by \(\lambda\) only changes the horizontal and vertical units together, the same quantity will control the Lipschitz constant of every \(f_{a,\lambda}\).

\begin{proposition}[Existence and scaling of the maximum slope]\label{prop:La}
For every \(a>0\),
\begin{equation}\label{eq:La-properties}
    1<L_a<\infty,
    \qquad
    \Lip(f_{a,\lambda})=L_a,
\end{equation}
and the supremum is attained at a finite point.  Furthermore,
\begin{equation}\label{eq:La-lower}
    L_a\geq h_a'(1)=\frac12+\frac a4,
\end{equation}
so \(L_a\to\infty\) as \(a\to\infty\).
\end{proposition}

\begin{proof}
Equation \eqref{eq:h-prime} is continuous and satisfies \(h_a'(t)\to1\) as \(t\to\infty\).  Hence it is bounded.  To prove that it exceeds one, set \(d(t)=t-h_a(t)>0\) for \(t>0\).  For any fixed \(t_0>0\), \(d(t_0)>0\), while \(d(t)\to0\).  For sufficiently large \(T\), \(d(T)<d(t_0)\), so the mean value theorem gives a point with \(d'(t)<0\), equivalently \(h_a'(t)>1\).  Since \(h_a'\to1\), a value strictly larger than one cannot be approached only at infinity; continuity therefore implies that the global maximum is attained on a compact interval.

Differentiating \(f_{a,\lambda}(s)=\lambda h_a(s/\lambda)\) gives
\(
f_{a,\lambda}'(s)=h_a'(s/\lambda)
\),
so its exact Lipschitz constant is \(L_a\), independently of \(\lambda\).  Finally, substituting \(t=1\) into \eqref{eq:h-prime} yields \eqref{eq:La-lower}.
\end{proof}

For numerical evaluation one may differentiate once more.  With \(\ell=\ell_a(t)\),
\begin{equation}\label{eq:h-second}
    h_a''(t)
    =a\ell(1-\ell)\bigl[2+at(1-2\ell)\bigr].
\end{equation}
Thus any interior maximizer of \(h_a'\) satisfies the scalar equation
\begin{equation}\label{eq:max-slope-root}
    2+at(1-2\ell_a(t))=0.
\end{equation}
No monotonicity of \(a\mapsto L_a\) is needed for the theory.  The rigorous conclusion relevant to the hard-threshold limit is the divergence supplied by \eqref{eq:La-lower}: increasingly sharp approximations necessarily admit increasingly large worst-case slopes.

\subsection{Exact scalar weak convexity}

The maximum slope also determines the nonconvexity of the induced penalty.  Recall that a function \(\psi\) on a Euclidean space is \(\rho\)-weakly convex when
\(
\psi(x)+\tfrac\rho2\norm{x}^2
\)
is convex.  The smallest valid \(\rho\) is precisely the quadratic curvature needed to compensate for the most concave part of \(\psi\).

\begin{theorem}[Exact weak-convexity constant]\label{thm:scalar-weak-convexity}
Define
\begin{equation}\label{eq:rho-def}
    \rho_a=1-\frac1{L_a}\in(0,1).
\end{equation}
Then \(p_a\) and \(\phi_{a,\lambda}\) are \(\rho_a\)-weakly convex, and no smaller constant is valid.  The constant is independent of \(\lambda\).
\end{theorem}

\begin{proof}
From \eqref{eq:phi-second} and Proposition \ref{prop:La},
\begin{equation}\label{eq:phi-curvature-lower}
    \phi_{a,\lambda}''(u)
    \geq\frac1{L_a}-1=-\rho_a.
\end{equation}
Therefore \(\phi_{a,\lambda}(u)+\rho_a u^2/2\) is convex.  Let \(t_a\) attain the maximum \(h_a'(t_a)=L_a\), and set \(s_a=\lambda t_a\) and \(u_a=f_{a,\lambda}(s_a)\).  Then
\[
    \phi_{a,\lambda}''(u_a)
    =\frac1{f_{a,\lambda}'(s_a)}-1
    =\frac1{L_a}-1=-\rho_a.
\]
Any smaller constant leaves negative curvature at \(u_a\).  The normalized case is identical, and the scaling identity
\(
\phi_{a,\lambda}''(u)=p_a''(u/\lambda)
\)
shows that \(\lambda\) cancels.
\end{proof}

Combining \eqref{eq:unit-objective-curvature} with the definition of \(L_a\) yields the exact geometric identity
\begin{equation}\label{eq:three-geometries}
    \boxed{\begin{aligned}
    \sup_{s\geq0}f_{a,\lambda}'(s)&=L_a,\\
    \inf_{u\geq0}\phi_{a,\lambda}''(u)&=-\rho_a,\\
    \inf_{u\geq0}G_s''(u)&=\frac1{L_a}.
    \end{aligned}}
\end{equation}
The three quantities are therefore not separate tuning parameters.  They express the same transition geometry as perturbation sensitivity, negative penalty curvature, and residual curvature of the unit proximal objective.  This common constant is what allows the scalar construction to be lifted to matrices without losing control of proximal uniqueness.

\section{Spectral Regularization and Scaled Proximal Maps}

\subsection{Matrix weak convexity}

Having analysed the scalar penalty, we now generalize it to a matrix penalty acting through the singular values.  We therefore set
\begin{equation}\label{eq:spectral-penalty}
    \Phi_{a,\lambda}(X)
    =\sum_{i=1}^q\phi_{a,\lambda}(\sigma_i(X)).
\end{equation}

\begin{theorem}[Exact matrix weak convexity]\label{thm:matrix-weak-convexity}
The spectral function \(\Phi_{a,\lambda}\) is \(\rho_a\)-weakly convex on \(\R^{M\times N}\), and \(\rho_a\) is the smallest valid constant.
\end{theorem}

\begin{proof}
Set
\(
r_{a,\lambda}(u)=\phi_{a,\lambda}(u)+\rho_a u^2/2
\).
Theorem \ref{thm:scalar-weak-convexity} makes \(r_{a,\lambda}\) convex.  It is also nondecreasing because
\begin{equation}\label{eq:r-prime}
    r_{a,\lambda}'(u)
    =f_{a,\lambda}^{-1}(u)-\frac{u}{L_a}\geq0.
\end{equation}
Indeed, \(f_{a,\lambda}^{-1}(u)\geq u\) and \(L_a>1\).  A convex nondecreasing function of nonnegative singular values generates a convex unitarily invariant function \cite{lewis1995}.  Since
\[
    \Phi_{a,\lambda}(X)+\frac{\rho_a}{2}\norm{X}_F^2
    =\sum_{i=1}^q r_{a,\lambda}(\sigma_i(X)),
\]
matrix weak convexity follows.

For sharpness, fix unit vectors \(b\in\R^M\) and \(c\in\R^N\), and restrict the matrix function to the rank-one line \(X=u bc^T\).  For \(u>0\), this restriction equals \(\phi_{a,\lambda}(u)\).  A smaller matrix weak-convexity constant would therefore give a smaller scalar constant in a neighbourhood of the scalar worst-curvature point, contradicting Theorem \ref{thm:scalar-weak-convexity}.
\end{proof}

\subsection{Arbitrarily scaled scalar proximal map}

Theorem \ref{thm:unit-prox-scalar} establishes the intended shrinker at unit scale.  In proximal gradient the scale \(\beta\) generally differs from one, so the corresponding scaled problem must be analysed separately.  Its single-valuedness is again controlled by the weak-convexity constant.

\begin{theorem}[Scaled scalar proximal map]\label{thm:scaled-scalar-prox}
Let \(\beta>0\) satisfy
\begin{equation}\label{eq:scaled-prox-condition}
    \beta\rho_a<1.
\end{equation}
Define
\begin{equation}\label{eq:Tbeta}
    T_\beta(u)
    =u+\beta\phi_{a,\lambda}'(u)
    =\beta f_{a,\lambda}^{-1}(u)+(1-\beta)u.
\end{equation}
Then \(T_\beta:[0,\infty)\to[0,\infty)\) is a strictly increasing bijection, and
\begin{equation}\label{eq:scaled-scalar-prox}
    u_\beta(s)
    :=\prox_{\beta\phi_{a,\lambda}}(s)
    =T_\beta^{-1}(s).
\end{equation}
The scalar objective is \((1-\beta\rho_a)\)-strongly convex, the solution is unique, and
\begin{equation}\label{eq:scaled-prox-bounds}
    0\leq u_\beta(s)\leq s.
\end{equation}
Moreover, \(u_\beta\) is increasing and globally Lipschitz with constant at most \((1-\beta\rho_a)^{-1}\).
\end{theorem}

\begin{proof}
The scaled objective
\(
J_s(u)=\tfrac12(u-s)^2+\beta\phi_{a,\lambda}(u)
\)
has curvature
\begin{equation}\label{eq:scaled-objective-curvature}
    J_s''(u)=1+\beta\phi_{a,\lambda}''(u)
    \geq1-\beta\rho_a>0.
\end{equation}
Its first-order equation is \(T_\beta(u)=s\), so we just need to show that \(T_\beta(u)\) is invertible to derive its unique solution.

Notice
\(
T_\beta'(u)=J_s''(u)\geq1-\beta\rho_a
\),
so \(T_\beta\) is strictly increasing, with \(T_\beta(0)=0\) and \(T_\beta(u)\geq u\to\infty\).  It is therefore bijective.  The equation
\(
s-u=\beta\phi_{a,\lambda}'(u)\geq0
\)
gives \eqref{eq:scaled-prox-bounds}.  Monotonicity follows from inversion, and the derivative lower bound for \(T_\beta\) gives the stated inverse Lipschitz bound.
\end{proof}

Using \eqref{eq:phi-prime-lambert}, the root equation can be evaluated as
\begin{equation}\label{eq:Tbeta-lambert}
    T_\beta(u)
    =u+\frac{\beta\lambda}{a}
    W_0\!\left(\frac{au}{\lambda}e^{a(1-u/\lambda)}\right)=s.
\end{equation}
The unique root lies in \([0,s]\).  Bisection is globally convergent and requires at most
\(
\lceil\log_2(s/\epsilon_{\mathrm{root}})\rceil
\)
iterations to reduce the bracket length below \(\epsilon_{\mathrm{root}}\) when \(s>0\).  A safeguarded Newton method may be used for acceleration.  Large exponential terms should be evaluated in the log domain.

\subsection{Matrix proximal map}

\begin{theorem}[Scaled spectral proximal map]\label{thm:scaled-matrix-prox}
Assume \(\beta\rho_a<1\), and let
\(
Y=U\diagop(s_1,\ldots,s_q)V^T
\)
with \(s_1\geq\cdots\geq s_q\geq0\).  Then
\begin{equation}\label{eq:scaled-matrix-prox}
    \prox_{\beta\Phi_{a,\lambda}}(Y)
    =U\diagop\bigl(u_\beta(s_1),\ldots,u_\beta(s_q)\bigr)V^T.
\end{equation}
The matrix proximal subproblem has a unique global minimizer.
\end{theorem}

\begin{proof}
Let \(X=P\diagop(x_1,\ldots,x_q)Q^T\), where \(x_1\geq\cdots\geq x_q\geq0\).  Expanding the objective gives
\begin{align*}
    \frac12\norm{X-Y}_F^2+\beta\Phi_{a,\lambda}(X)
    &=\frac12\sum_i x_i^2+\frac12\sum_i s_i^2
      -\inner{X}{Y}_F\\
    &\quad+\beta\sum_i\phi_{a,\lambda}(x_i).
\end{align*}
For fixed singular values \((x_i)\), von Neumann's trace inequality gives
\(
\inner{X}{Y}_F\leq\sum_i x_i s_i
\),
with equality when the singular vectors are aligned.  The problem therefore reduces to
\begin{equation}\label{eq:matrix-prox-separable}
    \sum_{i=1}^q
    \left[\frac12(x_i-s_i)^2+\beta\phi_{a,\lambda}(x_i)\right].
\end{equation}
Each term has the unique minimizer \(u_\beta(s_i)\).  Since \(u_\beta\) is increasing, the output remains ordered and is a valid singular-value vector.  Finally, Theorem \ref{thm:matrix-weak-convexity} implies that the full matrix objective is \((1-\beta\rho_a)\)-strongly convex, hence the matrix minimizer is unique even when \(Y\) has repeated singular values.
\end{proof}

\begin{corollary}[Designed unit-scale shrinker]\label{cor:unit-matrix-prox}
At \(\beta=1\), condition \eqref{eq:scaled-prox-condition} holds automatically because \(\rho_a<1\), and
\begin{equation}\label{eq:unit-matrix-prox}
    \prox_{\Phi_{a,\lambda}}(Y)
    =U\diagop\bigl(f_{a,\lambda}(s_i)\bigr)V^T.
\end{equation}
\end{corollary}

The matrix analysis has now supplied the computational primitive needed for recovery.  Given an arbitrary input matrix, its singular vectors are preserved and each singular value is updated by a unique monotone scalar solve, provided \(\beta\rho_a<1\).  We next embed this primitive in a data-consistency iteration for the general observation model \eqref{eq:observation-model}.

\section{Recovery under General Linear Measurements}

\subsection{Smooth data fidelity}

The gradient of the data term in \eqref{eq:main-objective} is
\begin{equation}\label{eq:data-gradient}
    \nabla g(X)=\cA^*(\cA(X)-y).
\end{equation}
Indeed, for any direction \(H\),
\[
    \left.\frac{d}{dt}g(X+tH)\right|_{t=0}
    =\inner{\cA(X)-y}{\cA(H)}
    =\inner{\cA^*(\cA(X)-y)}{H}_F.
\]
Moreover,
\begin{equation}\label{eq:gradient-lipschitz}
    \norm{\nabla g(X)-\nabla g(Z)}_F
    \leq L_{\cA}\norm{X-Z}_F,
\end{equation}
so
\begin{equation}\label{eq:descent-lemma}
    g(Z)\leq g(X)+\inner{\nabla g(X)}{Z-X}_F
    +\frac{L_{\cA}}2\norm{Z-X}_F^2.
\end{equation}

The preceding upper bounds are used for optimization.  Reconstruction guarantees use the lower bound in Assumption \ref{ass:restricted-stability}.

\subsection{Identifiability and restricted stability}

Whenever \(\widehat X-X_0\in\cE\), Assumption \ref{ass:restricted-stability} gives
\begin{equation}\label{eq:prediction-to-image}
    \norm{\widehat X-X_0}_F^2
    \leq\frac{1}{\kappa p}
    \norm{\cA(\widehat X-X_0)}_2^2.
\end{equation}
As discussed in Section~2, this condition is inferential rather than algorithmic: it converts measurement-domain control into image-domain control, but is not used in the descent analysis below.

\subsection{Proximal-gradient iteration}

At the current iterate \(X^k\), minimize the quadratic local model
\begin{equation}\label{eq:local-model}
\begin{aligned}
    Q_\eta(X;X^k)
    ={}&\inner{\nabla g(X^k)}{X-X^k}_F
       +\frac{1}{2\eta}\norm{X-X^k}_F^2
       +\alpha\Phi_{a,\lambda}(X).
\end{aligned}
\end{equation}
Thus an update is chosen from
\[
    X^{k+1}\in\argmin_X Q_\eta(X;X^k).
\]
We now derive this minimizer explicitly.  Set
\[
    G^k:=\nabla g(X^k)
    =\cA^*\bigl(\cA(X^k)-y\bigr),
    \qquad
    Y^k=X^k-\eta G^k.
\]
For an arbitrary trial matrix \(X\), the two terms in
\eqref{eq:local-model} that arise from the smooth part satisfy
\begin{equation*}
\begin{aligned}
    &\inner{G^k}{X-X^k}_F
      +\frac{1}{2\eta}\norm{X-X^k}_F^2                                      \\
    &\quad=\frac{1}{2\eta}\left[
        \norm{X-X^k}_F^2
        +2\eta\inner{G^k}{X-X^k}_F
      \right]                                                                  \\
    &\quad=\frac{1}{2\eta}\left[
        \norm{X-X^k}_F^2
        +2\eta\inner{G^k}{X-X^k}_F
        +\eta^2\norm{G^k}_F^2
      \right]
      -\frac{\eta}{2}\norm{G^k}_F^2                                      \\
    &\quad=\frac{1}{2\eta}
      \norm{X-X^k+\eta G^k}_F^2
      -\frac{\eta}{2}\norm{G^k}_F^2                                     \\
    &\quad=\frac{1}{2\eta}\norm{X-Y^k}_F^2
      -\frac{\eta}{2}\norm{G^k}_F^2.
\end{aligned}
\end{equation*}
The third equality uses the Frobenius identity
\(
\norm{A+B}_F^2=\norm{A}_F^2+2\inner{A}{B}_F+\norm{B}_F^2
\), with \(A=X-X^k\) and \(B=\eta G^k\).  Consequently,
\begin{equation*}
\begin{aligned}
    Q_\eta(X;X^k)
    =\frac{1}{2\eta}\norm{X-Y^k}_F^2
     +\alpha\Phi_{a,\lambda}(X)
     -\frac{\eta}{2}\norm{G^k}_F^2.
\end{aligned}
\end{equation*}
The final term depends on \(X^k\) but not on the optimization variable
\(X\), and therefore it does not affect the set of minimizers.  We obtain
\begin{equation*}
\begin{aligned}
    X^{k+1}
    &\in\argmin_X\left\{
       \frac{1}{2\eta}\norm{X-Y^k}_F^2
       +\alpha\Phi_{a,\lambda}(X)
      \right\}                                                                 \\
    &=\argmin_X\left\{
       \frac12\norm{X-Y^k}_F^2
       +\eta\alpha\Phi_{a,\lambda}(X)
      \right\}.
\end{aligned}
\end{equation*}
The second line follows by multiplying the whole objective by the positive
number \(\eta\), which preserves its minimizers.  By definition, the final
argmin is the possibly set-valued proximal mapping
\(
\prox_{\eta\alpha\Phi_{a,\lambda}}(Y^k)
\).  Under \((\eta\alpha)\rho_a<1\), Theorem
\ref{thm:scaled-matrix-prox} shows that this set contains exactly one matrix.
Substituting the formula for \(G^k\) therefore gives
\begin{equation}\label{eq:pg-update}
    X^{k+1}
    =\prox_{\eta\alpha\Phi_{a,\lambda}}
    \left(X^k-\eta\cA^*(\cA(X^k)-y)\right).
\end{equation}
More explicitly, if
\[
    Y^k
    =U^k\diagop(s_1^k,\ldots,s_q^k)(V^k)^T,
\]
then the unique solution of \eqref{eq:local-model} is
\begin{equation*}
    X^{k+1}
    =U^k\diagop\left(
       T_{\eta\alpha}^{-1}(s_1^k),\ldots,
       T_{\eta\alpha}^{-1}(s_q^k)
      \right)(V^k)^T.
\end{equation*}
Thus the proximal scale appearing in the update is
\begin{equation}\label{eq:beta-eta-alpha}
    \beta=\eta\alpha.
\end{equation}
Accordingly, the singular values are updated with the scaled map \(u_{\eta\alpha}=T_{\eta\alpha}^{-1}\) from Theorem \ref{thm:scaled-scalar-prox}.

\begin{algorithm}[t]
\caption{Smooth-to-Hard Spectral Proximal Gradient}
\label{alg:spectral-pg}
\begin{algorithmic}[1]
\REQUIRE \(y,\cA,\cA^*,a,\lambda,\alpha,\eta,X^0\), outer tolerance \(\epsilon_{\mathrm{out}}\), root tolerance \(\epsilon_{\mathrm{root}}\).
\STATE Set \(\rho_a=1-L_a^{-1}\) and verify \((\eta\alpha)\rho_a<1\).
\FOR{\(k=0,1,2,\ldots\)}
    \STATE \(R^k=\cA(X^k)-y\).
    \STATE \(Y^k=X^k-\eta\cA^*R^k\).
    \STATE Compute a thin SVD \(Y^k=U^k\diagop(s_1^k,\ldots,s_q^k)(V^k)^T\).
    \STATE For each \(i\), solve \(T_{\eta\alpha}(u_i^{k+1})=s_i^k\) on \([0,s_i^k]\).
    \STATE \(X^{k+1}=U^k\diagop(u_1^{k+1},\ldots,u_q^{k+1})(V^k)^T\).
    \IF{\(\norm{X^{k+1}-X^k}_F/\max\{1,\norm{X^k}_F\}\leq\epsilon_{\mathrm{out}}\)}
        \STATE stop.
    \ENDIF
\ENDFOR
\RETURN \(X^{k+1}\).
\end{algorithmic}
\end{algorithm}

\subsection{Step size and implementation}

We impose
\begin{equation}\label{eq:step-condition}
    0<\eta<\min\left\{\frac1{L_{\cA}},\frac1{\alpha\rho_a}\right\}.
\end{equation}
The first restriction makes the coefficient in \eqref{eq:sufficient-descent} strictly positive.  The second is exactly the scaled-prox uniqueness condition \((\eta\alpha)\rho_a<1\), which gives the well-posedness stated in Proposition \ref{prop:unique-update}.
When \(L_{\cA}\) is unavailable analytically, it can be estimated by a power iteration applied to \(\cA^*\cA\), or the first condition can be enforced by backtracking.

The dominant per-iteration cost is typically the SVD and the pair of operator calls \(\cA,\cA^*\).  Bisection of all singular-value roots costs
\(
O(q\log(1/\epsilon_{\mathrm{root}}))
\)
function evaluations and is normally secondary.  Truncated or randomized SVD may be useful in large problems, but the exact convergence theorem below applies to the exact proximal update; inexact variants require an explicit error analysis.

The construction of the algorithm is now complete: the gradient step enforces consistency with the measurements, and the spectral proximal step imposes the smooth-to-hard low-rank prior.  The two bounds in \eqref{eq:step-condition} ensure that these components are compatible.  We next show that the resulting iteration decreases the original objective and approaches first-order stationarity.

\section{Descent and Stationarity of the Recovery Iteration}

Fix \(a,\lambda,\alpha\) and a step size satisfying \eqref{eq:step-condition}, and suppose that each spectral proximal subproblem is solved exactly.  Since the regularizer is nonconvex, the appropriate first-order object is the limiting subdifferential \(\partial\), whose calculus for singular-value functions is described in \cite{lewissendov2005}.  A stationary point \(X^\star\) satisfies
\begin{equation}\label{eq:stationarity}
    0\in\nabla g(X^\star)+\alpha\partial\Phi_{a,\lambda}(X^\star).
\end{equation}

\begin{proposition}[Unique update]\label{prop:unique-update}
If \(\eta\alpha\rho_a<1\), the local model \eqref{eq:local-model} is
\begin{equation}\label{eq:local-strong-convexity}
    \left(\frac1\eta-\alpha\rho_a\right)\text{-strongly convex}.
\end{equation}
Consequently, \(X^{k+1}\) is uniquely defined.
\end{proposition}

\begin{proof}
The spectral function \(\Phi_{a,\lambda}\) is \(\rho_a\)-weakly convex by Theorem \ref{thm:matrix-weak-convexity}.  The quadratic term in \eqref{eq:local-model} contributes curvature \(1/\eta\), while its linear term contributes none.  The remaining curvature is \(1/\eta-\alpha\rho_a>0\).
\end{proof}

\begin{theorem}[Descent and stationarity]\label{thm:sufficient-descent}\label{thm:cluster-stationary}\label{thm:complexity}
For every \(k\),
\begin{equation}\label{eq:sufficient-descent}
    F_{a,\lambda}(X^{k+1})
    \leq F_{a,\lambda}(X^k)
    -\frac{1-\eta L_{\cA}}{2\eta}
    \norm{X^{k+1}-X^k}_F^2.
\end{equation}
Moreover,
\begin{equation}\label{eq:square-summability}
    \sum_{k=0}^{\infty}\norm{X^{k+1}-X^k}_F^2<\infty,
    \qquad
    \norm{X^{k+1}-X^k}_F\to0.
\end{equation}
Consequently,
\begin{equation}\label{eq:stationarity-residual}
    \distop\bigl(0,\partial F_{a,\lambda}(X^{k+1})\bigr)
    \leq\left(\frac1\eta+L_{\cA}\right)
    \norm{X^{k+1}-X^k}_F\longrightarrow0.
\end{equation}
Every cluster point of \(\{X^k\}\) satisfies \eqref{eq:stationarity}.  If the initial sublevel set
\begin{equation}\label{eq:initial-sublevel}
    \{X\st F_{a,\lambda}(X)\leq F_{a,\lambda}(X^0)\}
\end{equation}
is bounded, then the sequence has at least one cluster point and every cluster point is stationary.
Finally, with \(F_{\inf}=\inf_XF_{a,\lambda}(X)\), every integer \(K\geq1\) satisfies
\begin{equation}\label{eq:complexity}
\begin{aligned}
    \min_{0\leq k<K}
    \distop\bigl(0,\partial F_{a,\lambda}(X^{k+1})\bigr)^2
    \leq{}&
    \left(\frac1\eta+L_{\cA}\right)^2
    \frac{2\eta\bigl(F_{a,\lambda}(X^0)-F_{\inf}\bigr)}
         {(1-\eta L_{\cA})K}.
\end{aligned}
\end{equation}
\end{theorem}

\begin{proof}
Let \(D^k=X^{k+1}-X^k\).  Comparison of the local model at \(X^{k+1}\) and \(X^k\), followed by the descent lemma for \(g\), gives \eqref{eq:sufficient-descent}.  Summation proves \eqref{eq:square-summability}.  The update optimality condition supplies \(v^{k+1}\in\partial\Phi_{a,\lambda}(X^{k+1})\) such that
\[
  0=\nabla g(X^k)+\eta^{-1}D^k+\alpha v^{k+1}.
\]
Hence
\[
  \eta^{-1}(-D^k)+\nabla g(X^{k+1})-\nabla g(X^k)
  \in\partial F_{a,\lambda}(X^{k+1}),
\]
and Lipschitz continuity of \(\nabla g\) yields \eqref{eq:stationarity-residual}.

For a convergent subsequence, \eqref{eq:square-summability}, continuity of \(F_{a,\lambda}\), and closedness of the limiting-subdifferential graph give stationarity of its limit.  Boundedness of \eqref{eq:initial-sublevel} guarantees existence of such a subsequence.

Finally, telescope \eqref{eq:sufficient-descent} and use
\[
  F_{a,\lambda}(X^K)\geq F_{\inf}.
\]
Taking the smallest squared step and applying \eqref{eq:stationarity-residual} gives \eqref{eq:complexity}.
\end{proof}

The boundedness condition is not automatic because the penalty saturates and an undersampled \(\cA\) may have null directions.  Thus the theorem guarantees stationarity of cluster points, not convergence of the entire sequence.  It also gives \(O(\epsilon^{-1})\) iterations for a squared stationarity residual at most \(\epsilon\).

The optimization theory applies to every fixed pair \((a,\lambda)\), but it does not determine which pair should be used for a given data set.  The final part of the analysis separates this statistical choice from the convergence argument by evaluating training-based reconstructions on independent validation measurements.

\section{Data-Driven Selection by Independent Validation}

\subsection{Training--validation split}

Split the measurement coordinates into training and validation parts,
\begin{align}
    y_{\mathrm{tr}}&=\cA_{\mathrm{tr}}(X_0)+\varepsilon_{\mathrm{tr}},\label{eq:train-model}\\
    y_{\mathrm{val}}&=\cA_{\mathrm{val}}(X_0)+\varepsilon_{\mathrm{val}},
    \qquad
    \varepsilon_{\mathrm{val}}\sim N(0,\tau^2I_m),
    \label{eq:val-model}
\end{align}
with independent noise vectors.  Let
\(
\Theta_K=\{\theta_1,\ldots,\theta_K\}
\)
be a finite candidate set fixed before observing validation noise.  Usually \(\theta=(a,\lambda)\); \(\alpha\) may also be included if it is not fixed.  For each candidate, compute
\(
\widehat X_\theta=\widehat X_\theta(y_{\mathrm{tr}})
\)
using only the training data.

Define the validation prediction risk
\begin{equation}\label{eq:validation-risk}
    R_{\mathrm{val}}(\theta)
    =\frac1m\norm{\cA_{\mathrm{val}}(\widehat X_\theta-X_0)}_2^2
\end{equation}
and the debiased validation score
\begin{equation}\label{eq:validation-score}
    \widehat R_{\mathrm{val}}(\theta)
    =\frac1m\norm{\cA_{\mathrm{val}}(\widehat X_\theta)-y_{\mathrm{val}}}_2^2
    -\tau^2.
\end{equation}

\begin{proposition}[Conditional unbiasedness for a fixed candidate]\label{prop:validation-unbiased}
For every fixed \(\theta\in\Theta_K\),
\begin{equation}\label{eq:validation-unbiased}
    \E\left[
      \widehat R_{\mathrm{val}}(\theta)
      \mid y_{\mathrm{tr}},\cA_{\mathrm{val}}
    \right]
    =R_{\mathrm{val}}(\theta).
\end{equation}
\end{proposition}

\begin{proof}
Condition on \((y_{\mathrm{tr}},\cA_{\mathrm{val}})\) and set
\(
d_\theta=\cA_{\mathrm{val}}(\widehat X_\theta-X_0)
\).
Then \(d_\theta\) is fixed and
\(
\cA_{\mathrm{val}}(\widehat X_\theta)-y_{\mathrm{val}}
=d_\theta-\varepsilon_{\mathrm{val}}
\).  Since the validation noise is independent, centered, and has
\(\E\norm{\varepsilon_{\mathrm{val}}}_2^2=m\tau^2\),
\begin{align*}
    \E\left[
      \widehat R_{\mathrm{val}}(\theta)
      \mid y_{\mathrm{tr}},\cA_{\mathrm{val}}
    \right]
    &=\frac1m\norm{d_\theta}_2^2
      -\frac2m\inner{d_\theta}{\E[\varepsilon_{\mathrm{val}}]}
      +\frac1m\E\norm{\varepsilon_{\mathrm{val}}}_2^2-\tau^2\\
    &=R_{\mathrm{val}}(\theta),
\end{align*}
which proves \eqref{eq:validation-unbiased}.
\end{proof}

Because Proposition \ref{prop:validation-unbiased} is candidate-wise, selection requires a deviation bound that holds uniformly over the finite grid.

\subsection{Finite-grid oracle inequality}

Select
\begin{equation}\label{eq:selected-theta}
    \widehat\theta
    \in\argmin_{\theta\in\Theta_K}
    \norm{\cA_{\mathrm{val}}(\widehat X_\theta)-y_{\mathrm{val}}}_2^2.
\end{equation}
Subtracting the common constant \(m\tau^2\) does not change the minimizer.

\begin{assumption}[Uniform candidate scale]\label{ass:candidate-bound}
Conditionally on the training data and validation operator, there is a finite \(B\) such that
\begin{equation}\label{eq:candidate-bound}
    \frac1{\sqrt m}\norm{d_\theta}_2\leq B,
    \qquad \theta\in\Theta_K.
\end{equation}
\end{assumption}

For a finite candidate set, some conditional \(B\) always exists.  A useful uniform scale follows, for example, from
\[
    \frac1{\sqrt m}\norm{d_\theta}_2
    \leq
    \frac{\norm{\cA_{\mathrm{val}}}_{\mathrm{op}}}{\sqrt m}
    \norm{\widehat X_\theta-X_0}_F.
\]
Unlike restricted stability, this is an upper-scale condition used only for concentration.

\begin{theorem}[Finite-candidate oracle inequality]\label{thm:oracle}
Under Assumption \ref{ass:candidate-bound}, for every \(0<\delta<1\), with conditional validation-noise probability at least \(1-\delta\),
\begin{equation}\label{eq:oracle}
    R_{\mathrm{val}}(\widehat\theta)
    \leq
    \min_{\theta\in\Theta_K}R_{\mathrm{val}}(\theta)
    +4\sqrt2\,\tau B
      \sqrt{\frac{\log(2K/\delta)}{m}}.
\end{equation}
\end{theorem}

\begin{proof}
Expanding the validation loss gives
\[
  m^{-1}\norm{d_\theta-\varepsilon_{\mathrm{val}}}_2^2
  =R_{\mathrm{val}}(\theta)
  -2m^{-1}\inner{d_\theta}{\varepsilon_{\mathrm{val}}}
  +m^{-1}\norm{\varepsilon_{\mathrm{val}}}_2^2.
\]
The last term is common to all candidates, while the middle term is conditionally centered Gaussian with variance at most \(4\tau^2B^2/m\).  The Gaussian tail bound and a union bound therefore give
\[
    \max_{\theta\in\Theta_K}\abs{\frac2m\inner{d_\theta}{\varepsilon_{\mathrm{val}}}}
    \leq2\sqrt2\,\tau B
      \sqrt{\frac{\log(2K/\delta)}{m}}
    =:t_\delta
\]
with probability at least \(1-\delta\).  Comparing the selected candidate with \(\theta^\star\in\argmin_{\theta\in\Theta_K}R_{\mathrm{val}}(\theta)\) on this event yields \(R_{\mathrm{val}}(\widehat\theta)\leq R_{\mathrm{val}}(\theta^\star)+2t_\delta\), which is \eqref{eq:oracle}.
\end{proof}

The selection cost is \(O(\sqrt{\log K/m})\).
The theorem compares with the best member of the pre-specified finite grid, not with an optimizer over a continuous parameter space.  It does not assert recovery of a ``true'' parameter pair.

The oracle inequality controls prediction error only in the validation measurement domain.  To convert this guarantee into Frobenius reconstruction error, we now invoke the validation-operator instance of the restricted stability condition.

\subsection{From validation prediction to reconstruction error}

Conditionally on the training data, the image-domain statement is applied
whenever the realized candidate errors satisfy
\[
    \{\widehat X_\theta-X_0:\theta\in\Theta_K\}
    \subseteq\cE_{\mathrm{val}},
\]
where \(\cE_{\mathrm{val}}\) is the prescribed class in Assumption
\ref{ass:restricted-stability}.  The validation stability bound then holds
simultaneously for every candidate error on that training realization.

\begin{corollary}[Image-domain bound]\label{cor:image-bound}
Under the assumptions of Theorem \ref{thm:oracle}, the validation clause
of Assumption \ref{ass:restricted-stability}, and the candidate-error
containment above, with probability at least \(1-\delta\),
\begin{equation}\label{eq:image-bound}
\begin{aligned}
    \norm{\widehat X_{\widehat\theta}-X_0}_F^2
    \leq\frac1{\kappa_{\mathrm{val}}}
    \left[
      \min_{\theta\in\Theta_K}R_{\mathrm{val}}(\theta)
      +4\sqrt2\,\tau B
       \sqrt{\frac{\log(2K/\delta)}{m}}
    \right].
\end{aligned}
\end{equation}
\end{corollary}

\begin{proof}
Apply \eqref{eq:validation-restricted-stability} to
\(H=\widehat X_{\widehat\theta}-X_0\), and then substitute
\eqref{eq:oracle}.
\end{proof}

Proposition \ref{prop:validation-unbiased} concerns each candidate before minimization over the grid.  Consequently, the smallest observed validation score should not be reported as an unbiased estimate of the selected procedure's risk.  In practice one may select \(\widehat\theta\) from the split and then rerun the recovery algorithm on all measurements with this parameter fixed.  The oracle inequality above applies directly to the split-trained candidates, not automatically to the refitted estimator; transferring the guarantee through refitting would require an additional stability argument.  The same distinction must be maintained if cross-fitting is used to reduce dependence on a single split.

This completes the theoretical chain.  Independent validation selects among the stationary reconstructions produced by the optimization method, while restricted stability determines whether their validation-domain performance controls the underlying matrix error.  Neither ingredient substitutes for the other: convergence concerns computation, whereas restricted stability and validation concern identifiability and statistical selection.

\section{Numerical Evaluation}
\label{sec:numerical_evaluation}

\subsection{Objective and protocol overview}
\label{sec:evaluation_objective}

We apply the proposed recovery method to retrospectively undersampled cardiac cine MRI to evaluate its reconstruction performance and empirical robustness under different acceleration and noise conditions.

The baselines were Soft and Hard singular-value thresholding and low-rank plus sparse (L+S) reconstruction.  We compared them with the proposed smooth-to-hard regularizer over prespecified accelerations and measurement-noise levels.  The same formal OCMR series, sampling masks, noisy measurements, validation splits, and stopping rule were used for all four methods.  All retrospective degradation was applied after coil compression, frame-wise sensitivity estimation, and sensitivity-weighted coil combination; hence, the reconstruction problem was a single-image Cartesian inverse problem.

\subsection{Data and single-image preprocessing}
\label{sec:data_preprocessing}

Fully sampled multi-coil cardiac cine data were obtained from the OCMR data set \cite{chen2020ocmr}.  Four series were reserved for pipeline checks, candidate construction, and parameter-boundary auditing.  The disjoint formal cohort contained \(20\) series: \(10\) short-axis and \(10\) long-axis acquisitions, nine acquired at \(1.5\)~T and eleven at \(3\)~T, with \(16\)--\(27\) cardiac phases (median \(21.5\)).  Because public patient identifiers were unavailable, the series were treated as the statistical units.

Readout oversampling was removed and each series was brought to its reconstruction matrix.  Four formal series contained partial-Fourier readout; their missing samples were completed by eight fixed POCS iterations.  A global semi-unitary compression basis then reduced the original \(15\)--\(38\) receiver coils to \(12\) virtual coils, retaining \(97.68\%\)--\(99.29\%\) of the coil energy (median \(98.69\%\)).

ESPIRiT sensitivity maps were estimated independently for each cardiac frame from the fully sampled virtual-coil data, using the full readout extent, \(24\) central \(k_y\) calibration lines, a \(6\times6\) kernel, a relative calibration singular-value threshold of \(0.02\), and an image-domain eigenvalue threshold of zero.  The retained ESPIRiT eigenvector was normalized pixelwise so that \(\sum_{c=1}^{12}\abs{\widehat S_{t,c}(r)}^2=1\) on its support; hence the usual denominator in the sensitivity-weighted combination equals one.  The combined reference frame was
\begin{equation}
  x_{0,t}=\sum_{c=1}^{12}\overline{\widehat S_{t,c}}\,
  \mathcal F^{-1}_{xy}k_{t,c},
  \label{eq:coil_combination}
\end{equation}
where \(\widehat S_{t,c}\) and \(k_{t,c}\) are the sensitivity map and fully sampled \(k\)-space of virtual coil \(c\).  The resulting complex cine \(X_0\) was normalized by its root-mean-square amplitude,
\begin{equation}
  X_0\leftarrow X_0\bigg/
  \frac{\lVert X_0\rVert_F}{\sqrt{N_xN_yT}}.
  \label{eq:rms_normalization}
\end{equation}

\subsection{Retrospective acquisition model}
\label{sec:acquisition_model}

The numerical model is complex-valued.  We regard \(\mathbb C^{N_xN_y\times T}\) as a real Hilbert space with inner product
\begin{equation}
  \langle X,Y\rangle_{\mathbb R}
  =\operatorname{Re}\operatorname{tr}(X^*Y),
  \label{eq:complex_real_inner_product}
\end{equation}
and use the corresponding conjugate adjoint \(\cA^*\).  With this convention, the singular-value spectral functions, weak-convexity constants, and proximal-map results above are unchanged.

After coil combination, the forward operator was
\begin{equation}
  \cA_{\Omega}(X)=P_{\Omega}\mathcal F_{xy}X,
  \label{eq:single_image_forward}
\end{equation}
where \(\mathcal F_{xy}\) is the orthonormal two-dimensional Fourier transform applied frame by frame.  Acceleration factors were \(R\in\{4,6,8\}\).  Each frame retained \(\operatorname{round}(N_y/R)\) phase-encoding lines, including eight central ACS lines.  Peripheral lines followed a centre-weighted variable-density distribution with usage-dependent complementary temporal coverage; every selected line contained the complete readout direction.

Circular complex Gaussian noise was added to the acquired single-image \(k\)-space.  With \(s_{\Omega}=\cA_{\Omega}(X_0)\), \(m\) acquired complex samples, and \(z=(z_{\mathrm r}+\mathrm i z_{\mathrm i})/\sqrt2\) for independent standard-normal \(z_{\mathrm r}\) and \(z_{\mathrm i}\),
\begin{equation}
  y_{\Omega}=s_{\Omega}+\sigma_qP_{\Omega}z,
  \qquad
  \sigma_q=\frac{\lVert s_{\Omega}\rVert_2}
  {\sqrt m\,10^{q/20}},
  \qquad q=\mathrm{SNR}_{\mathrm{nom}}.
  \label{eq:noise_scale}
\end{equation}
Thus, \(q\in\{0,10,20,30\}\)~dB is an input to the noise-scaling formula, not the noise amplitude itself: \(0\)~dB produces the strongest noise and \(30\)~dB the weakest.  For each series and \(R\), the same standard noise draw was rescaled across the four values of \(q\).  Realized measurement SNR differed from the nominal value by at most approximately \(0.04\)~dB.

For validation noise \(\varepsilon=(\varepsilon_{\mathrm r}+\mathrm i\varepsilon_{\mathrm i})/\sqrt2\), with independent \(\varepsilon_{\mathrm r},\varepsilon_{\mathrm i}\sim N(0,\tau^2I_m)\), one has \(\E\norm{\varepsilon}_2^2=m\tau^2\), so the fixed-candidate conditional unbiasedness statement remains unchanged.  Moreover, for deterministic \(d\in\mathbb C^m\),
\begin{equation}
  \operatorname{Var}\!\left[\operatorname{Re}(d^*\varepsilon)\right]
  =\frac{\tau^2}{2}\norm{d}_2^2.
  \label{eq:complex_validation_variance}
\end{equation}
Consequently, the real finite-grid oracle bound in Section~8 remains valid for the complex experiment as a conservative bound; using the exact complex-Gaussian variance would replace its constant \(4\sqrt2\) by \(4\).

Table~\ref{tab:experimental_settings} summarizes the fixed protocol. Deterministic child seeds derived from the series identifier, acceleration, and purpose were generated from base seeds 2026100301 (masks), 2026100302 (validation), and 2026100303 (noise).

\begin{table}[t]
\centering
\caption{Fixed experimental settings. Candidate sets were constructed with the four pilot series and frozen before evaluation of the 20 formal series.}
\label{tab:experimental_settings}
\small
\renewcommand{\arraystretch}{1.08}
\begin{tabular}{@{}>{\raggedright\arraybackslash}p{0.23\textwidth}>{\raggedright\arraybackslash}p{0.70\textwidth}@{}}
\toprule
Component & Setting \\
\midrule
Formal data & 20 OCMR series; 10 short-axis/10 long-axis; nine 1.5-T/eleven 3-T; 16--27 phases. \\
Pilot data & Four disjoint series used only for pipeline checks, candidate construction, and boundary auditing. \\
Partial Fourier & Eight POCS iterations for four formal series before coil compression. \\
Coil processing & Global compression to 12 virtual coils; frame-wise ESPIRiT with 24 central $k_y$ lines, $6\times6$ kernel, calibration threshold $0.02$, image eigenvalue threshold $0$; Eq.~\eqref{eq:coil_combination}; complex RMS normalization. \\
Sampling & time-varying variable-density Cartesian masks; eight central ACS lines in every frame and included in the reported acceleration. \\
Noise & Eq.~\eqref{eq:noise_scale}; common standard draw across $q$ for fixed series and $R$. \\
Validation & 15\% of acquired peripheral line--time pairs; same split for all methods; ACS always retained in training. \\
Frozen candidates & Separate set for every method and $R\times q$ condition: six Proposed, four Soft, four Hard, and six L+S candidates. \\
Optimization & Relative iterate tolerance $10^{-5}$ or 300 iterations; Proposed scalar proximal-root tolerance $10^{-9}$. \\
\bottomrule
\end{tabular}
\end{table}

\subsection{Condition-specific parameter selection}
\label{sec:parameter_selection}

Candidate sets were allowed to differ by method, acceleration, and nominal SNR, but their construction used only the four pilot series.  For each formal series and condition, \(15\%\) of acquired peripheral line--time pairs formed a validation set; ACS samples remained in training.  All methods used the same split.  A candidate \(\theta\) was fitted only to the training data and scored by
\begin{equation}
  V(\theta)=\frac{1}{m_{\mathrm{val}}}
  \left\lVert\cA_{\mathrm{val}}(\widehat X_\theta)
  -y_{\mathrm{val}}\right\rVert_2^2.
  \label{eq:validation_score}
\end{equation}
The lowest-score candidate was selected independently for each series and condition and then refitted to all acquired samples.  Reference-image metrics were not used for selection.  As noted in Section~8.3, the oracle inequality applies directly to the split-trained candidates.  The full-data refit is an empirical post-selection step and is not covered without an additional algorithmic-stability argument.

For each fit, \(X_{\mathrm{init}}\) denotes the adjoint, or zero-filled, initialization computed from the measurements used in that fit.  For Proposed, \(\theta=(\alpha,a,c)\) and \(\lambda=c\sigma_1(X_{\mathrm{init}})\).  Soft and Hard used normalized thresholds \(\tau=q_s\sigma_1(X_{\mathrm{init}})\).  L+S used pairs \((q_L,q_S)\) with \(\tau_L=q_L\sigma_1(X_{\mathrm{init}})\) and \(\tau_S=q_S\max\abs{\mathcal F_tX_{\mathrm{init}}}\).  The retained sets were sparse collections of pilot-supported candidates rather than complete Cartesian products.  This keeps the search computationally tractable while defining a fixed, condition-specific selection procedure for every method.

\subsection{Outcome measures and paired analysis}
\label{sec:evaluation_metrics}

Magnitude PSNR was the primary summary measure.  Complex NMSE and magnitude SSIM were complementary measures used to assess whether the PSNR conclusions were consistent across error magnitude and structural fidelity.  SSIM was averaged frame-wise inside a reference-defined ROI whose temporal-mean magnitude exceeded \(5\%\) of its maximum.

All comparisons were paired by series and oriented so that positive values favoured Proposed:
\begin{equation}
\begin{aligned}
  \Delta\mathrm{NMSE}_B
  &=\mathrm{NMSE}_B-\mathrm{NMSE}_{\mathrm{P}},\\
  \Delta\mathrm{PSNR}_B
  &=\mathrm{PSNR}_{\mathrm{P}}-\mathrm{PSNR}_B,\\
  \Delta\mathrm{SSIM}_B
  &=\mathrm{SSIM}_{\mathrm{P}}-\mathrm{SSIM}_B.
\end{aligned}
  \label{eq:paired_differences}
\end{equation}
Condition-level summaries report the mean difference and number of positive differences; case-level distributions and 95\% nonparametric bootstrap intervals from 20,000 paired resamples are shown graphically. The intervals are descriptive and are not multiplicity-adjusted hypothesis tests.

\subsection{Reconstruction results}
\label{sec:evaluation_results}

Table~\ref{tab:single_image_psnr} gives the Proposed mean PSNR and its paired mean differences from each baseline. Proposed exceeded Hard in 11 of 12 conditions. The exception was $R=4$, 10~dB, where the mean difference was $-0.272$~dB (7/20 wins; 95\% bootstrap interval $[-0.506,-0.027]$~dB). Against Soft, Proposed was generally less favourable at 0--10~dB, mixed at 20~dB, and favourable at every 30-dB condition. Against L+S, Proposed was lower throughout 10~dB, mixed at 0 and 20~dB, and higher at every 30-dB condition.

\begin{table}[t]
\centering
\caption{Mean magnitude PSNR of Proposed and mean paired PSNR differences in dB (Proposed minus baseline), with Proposed wins in parentheses. $\mathrm{SNR}_{\mathrm{nom}}$ is the parameter $q$ in Eq.~\eqref{eq:noise_scale}; larger values mean less added noise. Each row contains 20 paired series.}
\label{tab:single_image_psnr}
\small
\setlength{\tabcolsep}{3.2pt}
\begin{tabular}{cccccc}
\toprule
$R$ & $\mathrm{SNR}_{\mathrm{nom}}$ & Proposed & $\Delta$ Hard & $\Delta$ Soft & $\Delta$ L+S \\
 & (dB) & PSNR (dB) & & & \\
\midrule
4 & 0  & 22.601 & $+2.664$ (20/20) & $-0.904$ (8/20)  & $+1.157$ (13/20) \\
4 & 10 & 29.547 & $-0.272$ (7/20)  & $-0.435$ (4/20)  & $-1.123$ (0/20) \\
4 & 20 & 37.271 & $+1.074$ (18/20) & $-0.082$ (6/20)  & $-0.329$ (2/20) \\
4 & 30 & 41.454 & $+3.386$ (20/20) & $+0.702$ (20/20) & $+0.623$ (18/20) \\
6 & 0  & 21.998 & $+4.222$ (20/20) & $-2.312$ (0/20)  & $+0.244$ (17/20) \\
6 & 10 & 28.581 & $+1.409$ (20/20) & $-0.937$ (0/20)  & $-0.814$ (0/20) \\
6 & 20 & 35.265 & $+0.798$ (19/20) & $+0.243$ (15/20) & $+0.184$ (14/20) \\
6 & 30 & 38.673 & $+1.876$ (20/20) & $+0.930$ (20/20) & $+1.062$ (19/20) \\
8 & 0  & 21.046 & $+3.183$ (16/20) & $-2.499$ (1/20)  & $-0.711$ (11/20) \\
8 & 10 & 27.518 & $+1.445$ (20/20) & $-0.625$ (1/20)  & $-0.708$ (2/20) \\
8 & 20 & 33.670 & $+0.488$ (19/20) & $+0.683$ (19/20) & $+0.749$ (18/20) \\
8 & 30 & 36.395 & $+1.003$ (20/20) & $+1.328$ (20/20) & $+1.777$ (20/20) \\
\bottomrule
\end{tabular}
\end{table}

Figure~\ref{fig:paired_psnr_distributions} shows that the 0-dB differences are wider and more skewed than those at higher nominal SNR. In particular, at $R=8$, 0~dB versus L+S, the mean was $-0.711$~dB but the median was $+0.336$~dB and the bootstrap interval crossed zero. The 0-dB row should therefore be interpreted as an extreme-noise regime with substantial case-to-case variability, rather than forced into a monotone trend.

\begin{figure}[t]
  \centering
  \includegraphics[width=\textwidth]{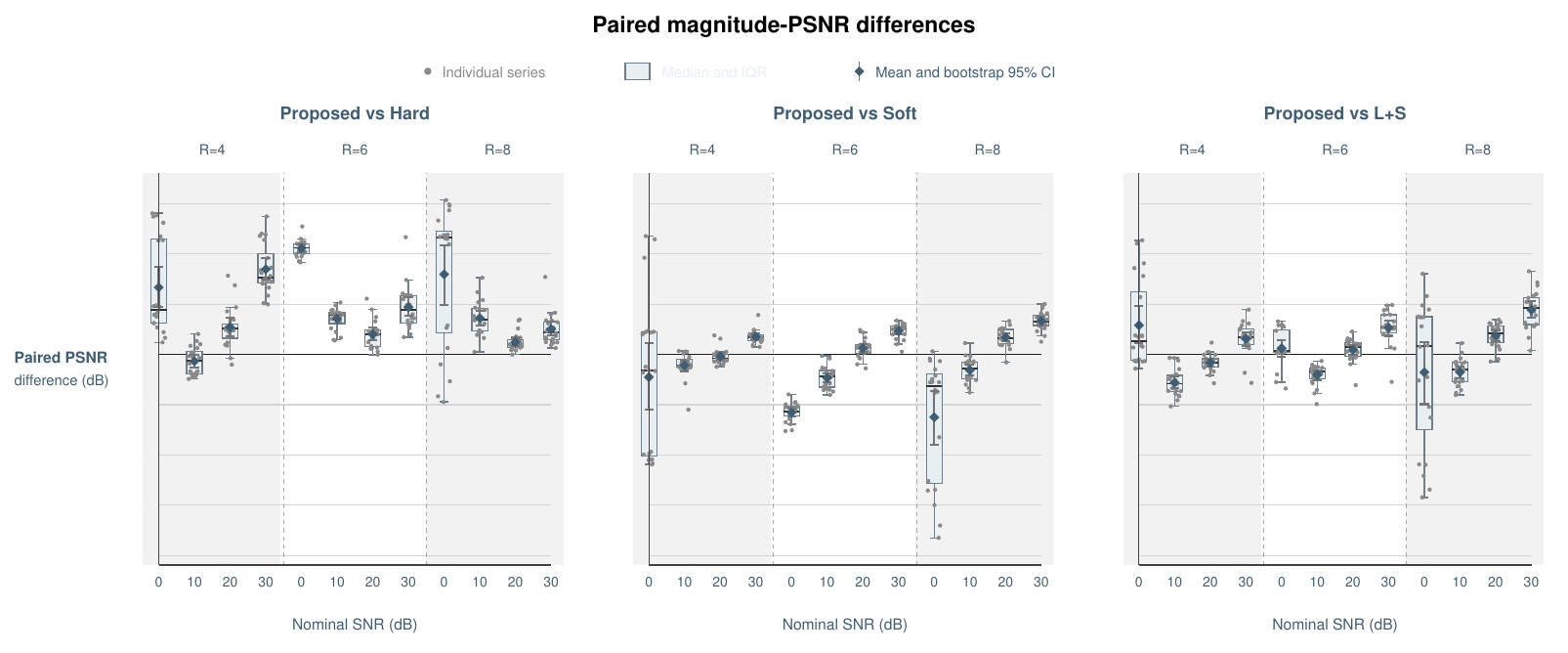}
  \caption{Paired magnitude-PSNR differences for the 20 formal series. Positive values favour Proposed. Individual series, medians, interquartile ranges, means, and 95\% paired-bootstrap intervals show both the average effect and its case-level dispersion.}
  \label{fig:paired_psnr_distributions}
\end{figure}

The complementary metrics in Fig.~\ref{fig:secondary_metrics} support the same condition-dependent interpretation. Relative to Hard, NMSE and SSIM were predominantly favourable or near neutral. Relative to Soft and L+S, their signs generally followed the PSNR crossover, while the widest intervals occurred under the strongest noise. These measures are supporting evidence rather than a second optimization target.

\begin{figure}[t]
  \centering
  \includegraphics[width=\textwidth]{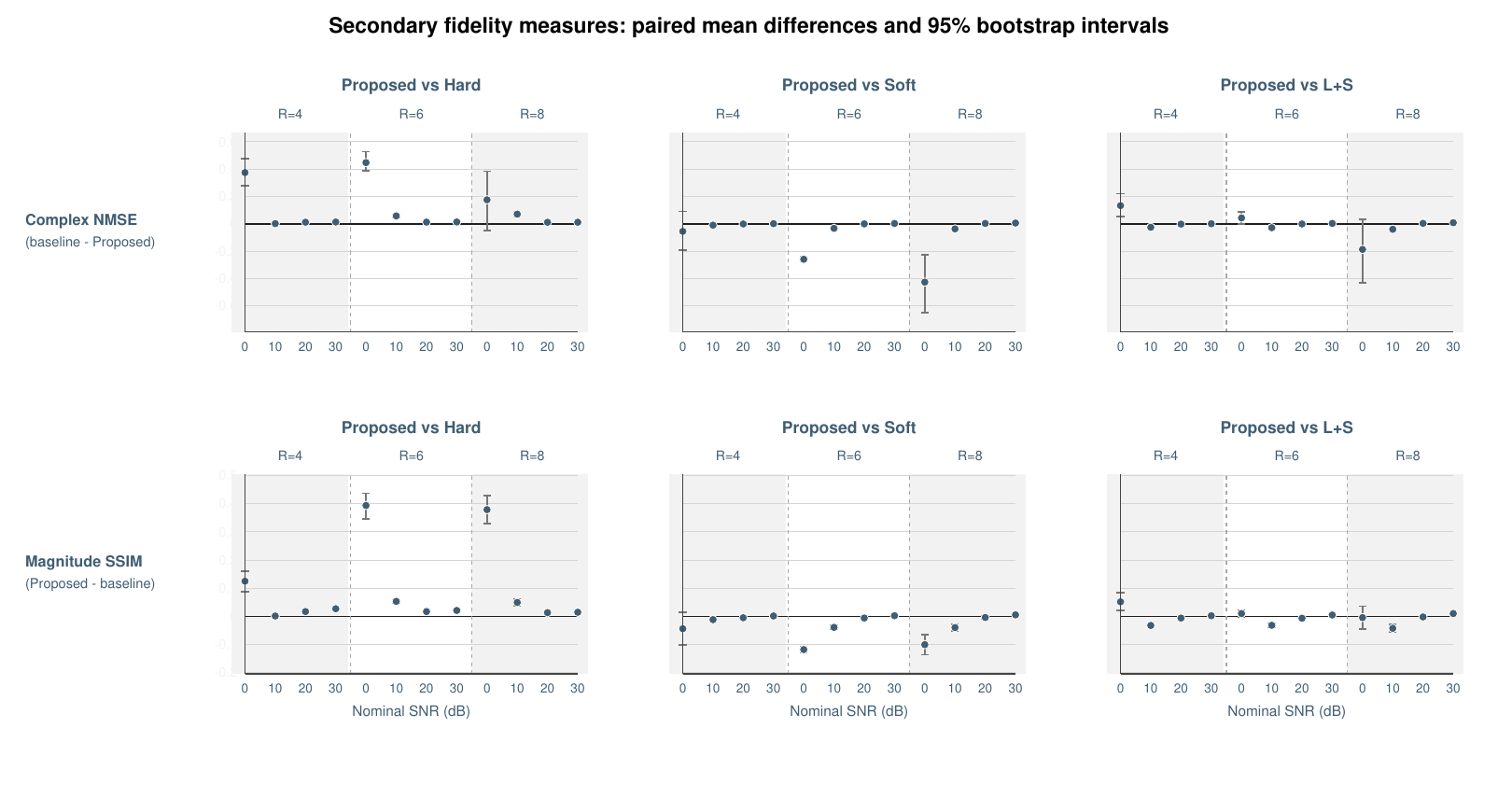}
  \caption{Mean paired differences and 95\% paired-bootstrap intervals for complex NMSE and magnitude SSIM. NMSE is baseline minus Proposed; SSIM is Proposed minus baseline, so positive values favour Proposed in both rows.}
  \label{fig:secondary_metrics}
\end{figure}

Figure~\ref{fig:qualitative_temporal} presents two prespecified median-difference examples: Proposed versus Soft at $R=4$, 0~dB, and Proposed versus L+S at $R=8$, 30~dB. Each block contains the reference, the displayed-frame sampling mask, the noisy zero-filled reconstruction, all four reconstructions, and jet-coloured absolute-error maps. Error cells under the reference and mask are intentionally blank. The examples illustrate the severe-noise stability of Soft and the reduced clean-data bias of Proposed without treating two cases as population-level evidence.

\begin{figure}[!htb]
  \centering
  \includegraphics[width=\textwidth]{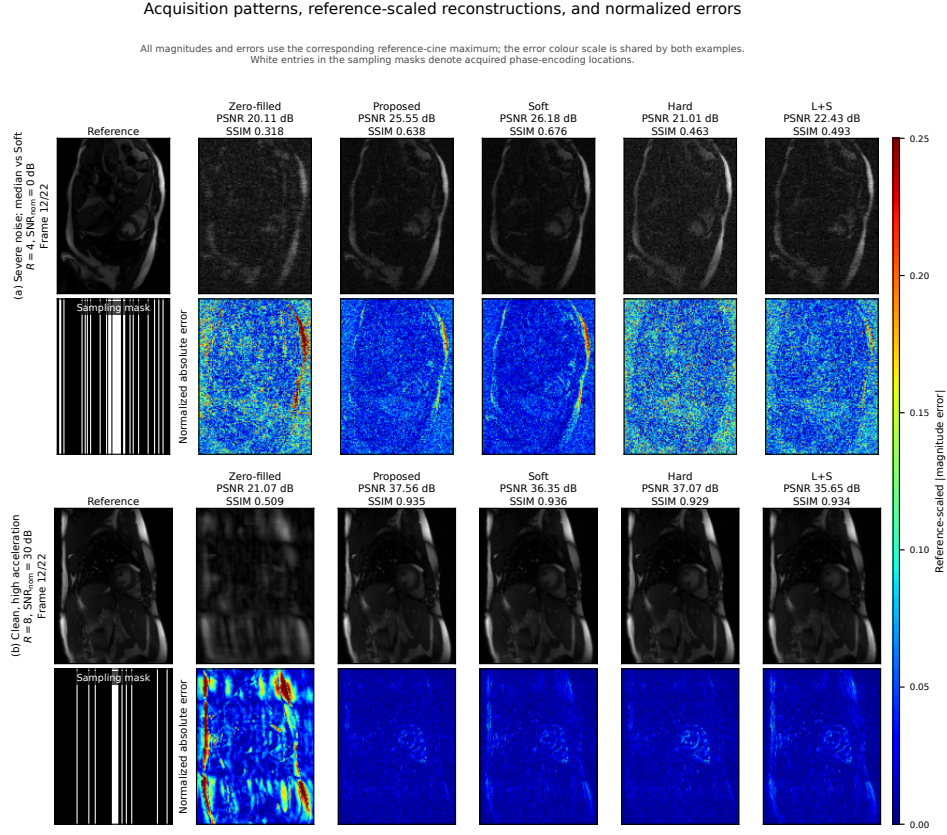}
  \caption{Acquisition patterns, reference-scaled reconstructions, and normalized magnitude-error maps for two prespecified median-difference examples: (a) Proposed versus Soft, $R=4$, $\mathrm{SNR}_{\mathrm{nom}}=0$~dB; (b) Proposed versus L+S, $R=8$, $\mathrm{SNR}_{\mathrm{nom}}=30$~dB. White entries in the masks denote acquired phase-encoding locations. Frame 12 of 22 is shown; PSNR and SSIM are computed over the complete cine. Magnitudes and errors are scaled by the corresponding reference-cine maximum, and all error maps share a colour scale.}
  \label{fig:qualitative_temporal}
\end{figure}

\subsection{Protocol diagnostics}
\label{sec:protocol_diagnostics}

\subsubsection{Pilot-only parameter-boundary audit}
\label{sec:boundary_audit}

Candidate ranges were expanded only when a pilot-series validation winner lay on a lower or upper boundary and its improvement over the adjacent interior candidate exceeded the numerical tolerance
\begin{equation}
  \epsilon_{\mathrm{num}}
  =10^{-8}+10^{-4}\abs{V_{\mathrm{boundary}}}.
  \label{eq:boundary_tolerance}
\end{equation}
Specifically, expansion was triggered only if
\begin{equation}
  V_{\mathrm{interior,best}}-V_{\mathrm{boundary}}
  >\epsilon_{\mathrm{num}}.
  \label{eq:boundary_trigger}
\end{equation}
The corresponding coordinate was moved outward while the remaining coordinates were held fixed, and the new candidates were evaluated with unchanged masks, noise, validation splits, and stopping rules.  Six rounds evaluated \(97\) outward settings, corresponding to \(388\) pilot reconstructions, with no failures; the sixth round produced no further trigger.  The candidate sets were then frozen before formal evaluation.  No boundary-driven search was performed on the \(20\) formal series, because doing so would adapt the tuning space to the evaluation data.  Thus, no further expansion was indicated by the prespecified audit under the frozen \(300\)-iteration protocol; changing the stopping or selection protocol would require the audit to be repeated.

\subsubsection{Convergence and iteration-cap sensitivity}
\label{sec:convergence_sensitivity}

For Proposed, Soft, and Hard, the relative iterate change was
\begin{equation}
  r_k=\frac{\norm{X^{k+1}-X^k}_F}
  {\max\{1,\norm{X^k}_F\}}.
  \label{eq:relative_iterate_change}
\end{equation}
For L+S, it was
\begin{equation}
  r_k=\frac{\sqrt{\norm{L^{k+1}-L^k}_F^2+\norm{S^{k+1}-S^k}_F^2}}
  {\max\!\left\{1,\sqrt{\norm{L^k}_F^2+\norm{S^k}_F^2}\right\}}.
  \label{eq:relative_iterate_change_ls}
\end{equation}
A run stopped when \(r_k\leq10^{-5}\) or after \(300\) iterations.  The relative tolerance was reached within \(300\) iterations in \(74/240\) Proposed, \(175/240\) Soft, \(238/240\) Hard, and \(40/240\) L+S formal refits; otherwise the run ended at the common cap.  The primary comparison is therefore between prespecified early-stopped procedures.

For each diagnostic condition and baseline comparison, the representative series was the one whose Proposed-minus-baseline PSNR difference was closest to the condition median; ties were resolved by series identifier.  The four corner conditions \(R\in\{4,8\}\) and \(\mathrm{SNR}_{\mathrm{nom}}\in\{0,30\}\)~dB were prespecified representative diagnostics.  The \(R=4\), \(10\)-dB Proposed--Hard continuation was a targeted post hoc sensitivity analysis prompted by the anomalous formal-cohort result.

With selected parameters held fixed, representative runs were continued to \(600\) iterations.  For Proposed and L+S in the four corner conditions, all eight runs reached tolerance by \(600\) iterations, the largest absolute PSNR change was \(0.104\)~dB, and no Proposed--L+S ordering changed.  In the post hoc \(R=4\), \(10\)-dB Proposed--Hard comparison, the gap changed from \(-0.139\) to \(-0.855\)~dB because Proposed decreased by \(0.716\)~dB while Hard remained unchanged after converging at iteration \(140\).  The condition in which Hard had higher mean PSNR is therefore not explained by an iteration cap that favoured Hard.  Because changing the cap alters the coupled tuning-and-stopping procedure, \(300\) iterations was retained for the primary experiment; a \(600\)-iteration primary protocol would require renewed pilot tuning and boundary auditing.

\begin{figure}[!htb]
  \centering
  \includegraphics[width=\textwidth]{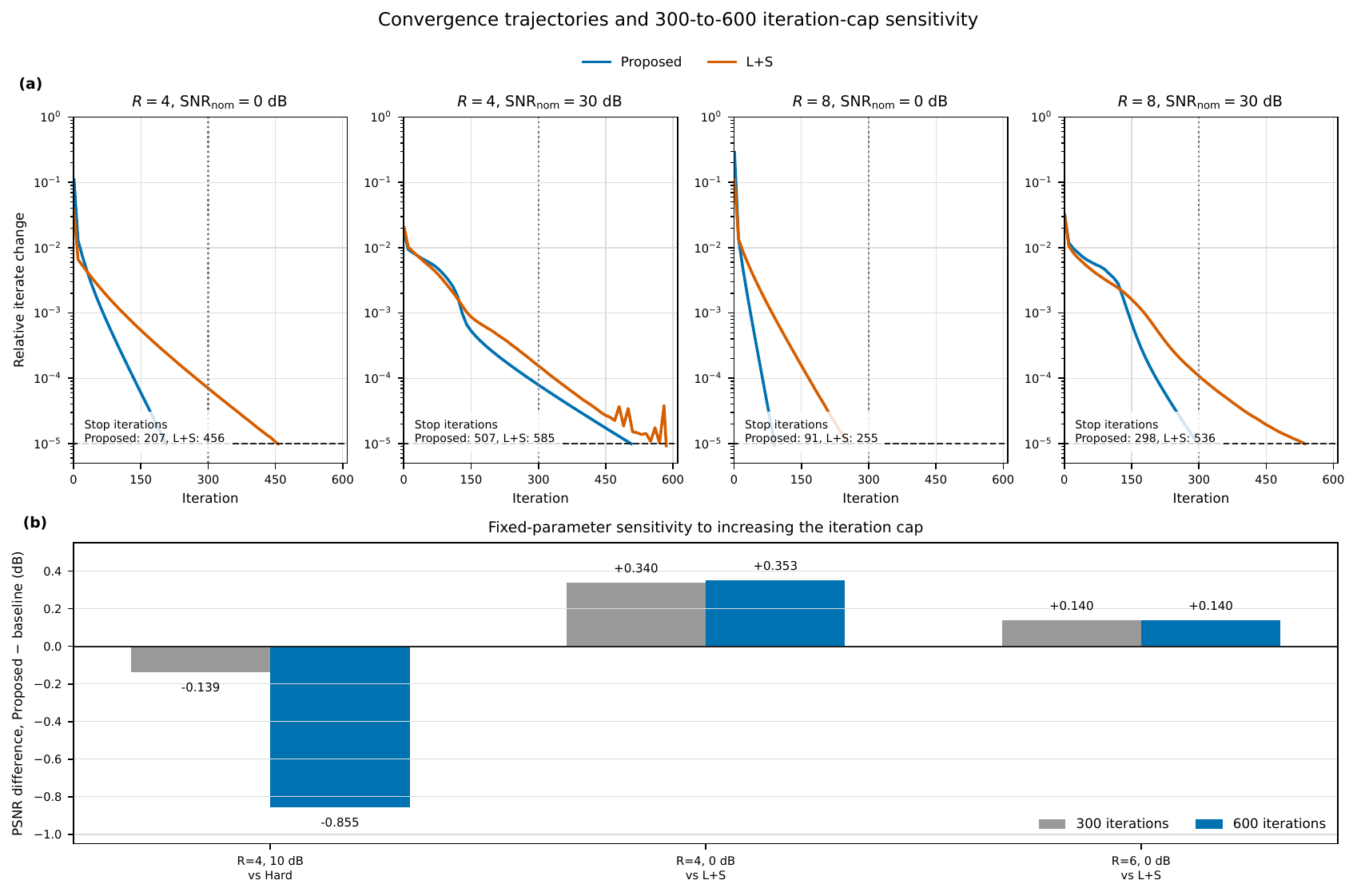}
  \caption{Convergence and iteration-cap sensitivity. Relative iterate changes are shown for representative Proposed and L+S corner conditions, together with Proposed-minus-baseline PSNR differences at 300 and 600 iterations for targeted comparisons. Parameters were held fixed during continuation.}
  \label{fig:convergence_sensitivity}
\end{figure}

\subsection{Discussion and limitations}
\label{sec:evaluation_discussion}

The observed noise-dependent crossover is consistent with reduced shrinkage bias in dominant temporal components when measurement noise is small, but the experiment does not isolate that mechanism causally.  In the \(0\)-dB regime, the acquired-signal and added-noise RMS values are equal; anatomy, sampling, and validation selection therefore produce broad and skewed paired differences.

The study contains \(20\) OCMR series, and patient-level independence could not be confirmed.  Undersampling and Gaussian noise were retrospective, and ESPIRiT maps and references were derived from fully sampled data before degradation.  Coil combination also limits the conclusions to the stated single-image protocol rather than a sensitivity-encoded multi-coil reconstruction.  Condition-specific candidate sets were indexed by the known nominal SNR used in the simulation.  In prospective data this quantity is generally unavailable and would require a noise/SNR estimate or replacement by a pooled SNR-agnostic candidate set.  Candidate sets were constructed from four pilot series, and the convergence analysis used representative cases rather than the complete cohort.  Finally, image-fidelity metrics were not linked to a blinded reader or clinical task.  Within these limits, the evidence supports a mean-PSNR advantage over Hard in most tested conditions and a noise-dependent, rather than universal, advantage over Soft and L+S.

\section{Conclusion}

This paper gives an explicit variational construction for the logistic smooth-to-hard shrinker.  Its Lambert-\(W\) inverse induces a nonnegative saturating penalty whose unique unit-scale proximal map is the prescribed shrinkage rule and whose hard-threshold limit is explicit.  The maximum shrinker slope \(L_a\) determines the exact weak-convexity constant \(\rho_a=1-L_a^{-1}\), linking perturbation sensitivity, negative curvature, and the residual curvature of the proximal objective.

The same constant is inherited by the spectral penalty.  At an arbitrary proximal scale \(\beta\), the update is obtained by singular-vector alignment and a monotone scalar solve, and it is unique when \(\beta\rho_a<1\).  Under the stated step-size condition, proximal gradient provides sufficient descent, vanishing stationarity residuals, stationary cluster points, and an \(O(1/K)\) bound for the smallest squared residual.  Independent validation over a finite grid controls prediction risk, while restricted stability is separately required for Frobenius reconstruction guarantees.

In the OCMR study, Proposed exceeded Hard in mean PSNR in \(11\) of \(12\) conditions and performed best relative to Soft and L+S in the \(30\)-dB regime, confirming a condition-dependent rather than universal advantage.  The conclusions remain limited by the small retrospective single-image cohort, simulated degradation, SNR-indexed candidate sets, early-stopped optimization, finite-\(a\) effective rather than exact rank reduction, post-selection refitting outside the direct oracle guarantee, and absence of clinical endpoints.

\end{document}